\documentclass[a4paper,11pt]{article}

\usepackage{geometry}
\usepackage[english]{babel}
\usepackage{amsmath}
\usepackage[utf8]{inputenc}
\usepackage[T1]{fontenc}
 
\usepackage{amsthm}
\usepackage{amssymb}
\usepackage[all]{xy}
\usepackage{MnSymbol}
\usepackage[,colorlinks=true,citecolor= DarkBlue,linkcolor=Black]{hyperref}
\usepackage[svgnames,table]{xcolor}

\newcommand\ddt[1]{\left .\frac{\d}{\d t}\right |_{#1}}

\theoremstyle{definition}
\newtheorem{dfn}{Definition}[section]

\theoremstyle{plain}
\newtheorem{thm}{Theorem}
\newtheorem*{thm*}{Theorem}
\newtheorem{prop}[dfn]{Proposition}
\newtheorem*{prop*}{Proposition}
\newtheorem{lem}[dfn]{Lemma}
\newtheorem*{lem*}{Lemma}

\theoremstyle{remark}
\newtheorem{rem}[dfn]{Remark}

\newcommand\chech[1]{\mathaccent"7014{#1}}

\newcommand{\R}{\mathbf{R}}

\newcommand{\N}{\mathbf{N}}

\newcommand{\g}{\mathfrak{g}}  
\newcommand{\h}{\mathfrak{h}}  
\newcommand{\s}{\mathfrak{s}}  
\renewcommand{\t}{\mathfrak{t}}
\newcommand{\p}{\mathfrak{p}}

\renewcommand{\L}{\mathcal{L}}
\newcommand{\e}{\mathrm{e}}
\renewcommand{\d}{\mathrm{d}}

\DeclareMathOperator{\GL}{GL}
\DeclareMathOperator{\Diff}{Diff}
\DeclareMathOperator{\Hom}{Hom}
\DeclareMathOperator{\Isom}{Isom}
\DeclareMathOperator{\Ad}{Ad}  

\DeclareMathOperator{\Stab}{Stab}
\DeclareMathOperator{\Ker}{Ker}
\DeclareMathOperator{\Lie}{Lie}
\DeclareMathOperator{\Vect}{Span}
\DeclareMathOperator{\Aut}{Aut}
\DeclareMathOperator{\id}{id}
\DeclareMathOperator{\Kill}{Kill}
\DeclareMathOperator{\Int}{Int}

\newcommand{\X}{\mathbf{X}}
\newcommand{\Ein}{\mathbf{Ein}}

\renewcommand{\epsilon}{\varepsilon}
\renewcommand{\geq}{\geqslant}
\renewcommand{\leq}{\leqslant}
\renewcommand{\hat}{\widehat}  
\newcommand{\hx}{\hat{x}}
\renewcommand{\tilde}{\widetilde}
\renewcommand{\bar}{\overline}

\author{Vincent Pecastaing}
\title{On two theorems about local automorphisms of geometric structures}

\begin{document}

\maketitle

\begin{abstract}
This article investigates a few questions about orbits of local automorphisms in manifolds endowed with rigid geometric structures. We give sufficient conditions for local homogeneity in a broad class of such structures, namely Cartan geometries, extending a classical result of Singer about locally homogeneous Riemannian manifolds. We also revisit a strong result of Gromov which describes the structure of the orbits of local automorphisms of manifolds endowed with $A$-rigid structures, and give a statement of this result in the setting of Cartan geometries.
\end{abstract}

\tableofcontents

\section{Introduction}
\label{s:intro}
The aim of this article is to revisit two classical results on local orbits of rigid geometric structures : Singer's homogeneity theorem for Riemannian manifolds and Gromov's result on $Is^{loc}$-orbits in a manifold endowed with a rigid geometric structure. We present a new generalization to Cartan geometries of Singer's theorem and an alternative and elementary approach of Gromov's theorem in the same geometric context. These two results are based on the study of the behaviour of the curvature and its derivatives up to some \textit{fixed} finite order.

\subsection{Around Singer's theorem}

Singer defined an infinitesimally homogeneous Riemannian manifold of order $k$ as a Riemannian manifold such that for every two points $x,y$, there exists a linear isometry sending $(\nabla^m R)_x$ to $(\nabla^m R)_y$ for all $m\leq k$, where $\nabla^m R$ denotes the $m$-th covariant derivative of the Riemannian curvature tensor $R$ with respect to the Levi-Civita connection $\nabla$. He proved the following result.

\begin{thm*}[\cite{singer}]
Let $(M^n,g)$ be a Riemannian manifold. There exists an integer $k_M \leq \frac{n(n-1)}{2}$ such that if $(M,g)$ is infinitesimally homogeneous of order $k_M$, then it is locally homogeneous.
\end{thm*}

If the manifold was locally homogeneous, then the existence of a local isometry sending $x$ to $y$ would imply infinitesimal homogeneity at \textit{any} order $k$. Thus, Singer not only proved that the converse holds but that infinitesimal homogeneity of order at most $\frac{n(n-1)}{2}$ is enough to ensure local homogeneity.

Following Singer, several generalizations of this theorem to other geometric structures have been proved by various people. In \cite{podesta}, Podesta and Spiro extended Singer's result to pseudo-Riemannian manifolds. Then, Opozda proposed a version of Singer's theorem for analytic affine connections (\cite{opozda1}) and also for $G$-structures endowed with a connection (\cite{opozda2}). However, her results do not exhibit an analogous bound on the order of derivatives in \cite{opozda1} and one has to check some additional algebraic condition in \cite{opozda2}. In a recent paper (\cite{gilkey}), P. Gilkey and al. gave a formulation in the setting of homothety homogeneity.

In the present paper, we propose a generalization of Singer's theorem to manifolds endowed with Cartan geometries. A Cartan geometry on a manifold is a rigid geometric structure, in the sense that everywhere, its infinitesimal automorphisms form a finite dimensional Lie algebra. This family of geometric structures includes most classical structures, such as pseudo-Riemannian metrics, conformal structures in dimension greater than $2$, projective structures or affine connections. Section \ref{s:background} presents some basic definitions and properties for the reader not familiar with Cartan geometries.

In a Cartan geometry, there exists a notion of curvature and of derivatives of the curvature. Formally, a Singer-like result in this context would be to associate to each type of Cartan geometry a bound $k_0$ such that if the curvature and all its derivatives up to order $k_0$ are somehow constant, then the geometry is locally homogeneous (see Definitions \ref{dfn:local_automorphism} and \ref{dfn:infinitesimal_homogeneity} for local and infinitesimal homogeneity). This is the content of our first main result : 

\begin{thm}
Let $(M,\mathcal{C})$ be a Cartan geometry on a connected manifold $M$, modeled on a homogeneous space $\X = G/P$. Set $d = \dim P$. If the geometry is $d$-infinitesimally homogeneous, then it is locally homogeneous.
\label{thm:singer}
\end{thm}

In fact, if $(M,\mathcal{C})$ is modeled on $G/P$, then $\dim M = \dim G/P$. Therefore, the integer $d$ only depends on $\dim M$ and the \textit{model space} $\X$ of the Cartan geometry we are dealing with. 

As we shall later see, in the general case, our definition of infinitesimal homogeneity cannot be directly interpreted on the manifold, but on some principal fiber-bundle over $M$. However, Theorem \ref{thm:singer} completely contains Singer's original theorem and its generalization to pseudo-Riemannian metrics essentially because the curvature of the Cartan geometry associated to a metric canonically corresponds to the Riemann curvature tensor of the metric. In fact, as soon as the model $G/P$ is reductive, infinitesimal homogeneity in the sense of Definition \ref{dfn:infinitesimal_homogeneity} is a tensorial condition that can be read on the manifold $M$ (see \cite{sharpe} p.197). For instance, if $(M^n,\nabla)$ is an affine manifold, we get the expected result : in that case, Theorem \ref{thm:singer} implies that if for every two points $x,y$ in $M$, and every $s \leq n^2$, there exists a linear transformation $f : T_xM \rightarrow T_y M$ such that $f^*(\nabla^s R^{\nabla})_y = (\nabla^s R^{\nabla})_x$, then $(M,\nabla)$ is locally homogeneous.

In addition to being a formulation of Singer's result in a larger geometric context, Theorem \ref{thm:singer} gives a short simple proof of the following result, adapted to Cartan geometries.

\begin{thm*}[Dense orbit theorem, \cite{gromov}, 3.3.A]
\label{thm:dense_orbit}
Let $(M,\mathcal{C})$ be a Cartan geometry modeled on $\X=G/P$. Assume the adjoint action of $P$ on $\g$ algebraic. If the pseudo-group $\Aut^{loc}(M)$ admits a dense orbit, then there exists an open dense subset of $M$ which is locally homogeneous.
\end{thm*}

Usually, this theorem is presented as a corollary of a more precise result which describes almost everywhere the structure of $\Aut^{loc}$-orbits. It is the object of the second half of this paper.

\subsection{Gromov's theorem}

The theorem of Gromov we are interested in says that given a manifold $M$ endowed with an $A$-rigid geometric structure, there exists an open dense set of $M$ on which the orbits of local isometries are closed submanifolds.

\begin{thm*}[\cite{gromov}]
Let $(M,\phi)$ be a manifold endowed with a rigid geometric structure of algebraic type $\phi$. Then, there exists an open dense subset on which orbits of local isometries of $\phi$ are closed submanifolds.
\end{thm*}

Roughly speaking, Gromov's proof consists in considering, for some $r$, orbits of ``isometric $r$-jet'', called $I^r$-orbits, that are in fact sets of points on which the $r$-jet of the structure $\phi$ is the same everywhere. He proved that outside an exceptional set with empty interior, these orbits are the fibers of a submersion. The key point is then a result of integrability (a ``Frobenius theorem'') which stipulates that if the order $r$ is chosen large enough, two points close enough in the same $I^r$-orbit are related by a local isometry of the structure, proving that $I^r$-orbits coincide locally with orbits of local isometries. 

Gromov's proof, even simplified in \cite{benoist} or \cite{feres}, remains technical and difficult, even when the structure is as elementary as the data of a global frame.

Recently, Melnick proposed an analogous theorem of integrability in the setting of Cartan geometries (\cite{melnick}). The point of view is closer to Singer's approach and generalizes results that Nomizu formulated in the Riemannian setting (\cite{nomizu}). Melnick proved for instance that if $(M,\mathcal{C})$ is a compact analytic Cartan geometry, there exist $r_0$ such that two points are related by a local automorphism of the Cartan geometry if and only if the first $r_0$ covariant derivatives of the \textit{curvature} are equal at these points (\cite{melnick}, Proposition 3.8).

In the smooth case, there is not such a global statement and it is more relevant to look at the \textit{tangent directions} in which the first $r$ derivatives of the curvature are constant, namely ``Killing generators of order $r$'' in Melnick's terminology. The question is then : does there exist an $r_0$ such that any Killing generator of order $r_0$ can be integrated into a local Killing field ? She answered positively to it when we restrict the study to an open dense subset (\cite{melnick}, Theorem 6.3). However, whereas Melnick's proof is convincing for analytic structures, there seems to be a significant gap in the proof of her general theorem (\cite{melnick}, Theorem 6.3).

In the second half of this article, we propose a new approach of Melnick's result in the smooth case. Although we reuse most definitions of \cite{melnick} and give a similar statement, which is the content of the following Theorem \ref{thm:frobenius}, our proof is noticeably different and the existence of local Killing field is derived from the classical Frobenius theorem.

\begin{thm}[Integrability of Killing generators]
\label{thm:frobenius}
Let $(M,\mathcal{C})$ be a Cartan geometry modeled on $\X=G/P$. Set $m = \dim G$, note $\pi : \hat{M} \rightarrow M$ the associated $P$-principal bundle and $\omega$ the Cartan connection. There exists an open dense subset $\Omega \subset M$ such that for all $\hx \in \pi^{-1}(\Omega)$ and $A \in \Kill^{m+1}(\hx)$ a Killing generator of order $m+1$ at $\hx$, there exists a local Killing field on a neighbourhood of $x$ whose lift $\hat{A}$ satisfies $\omega(\hat{A}(\hx))=A$.
\end{thm}

Note that this result not only says that, on an open dense set, any Killing generator of a fixed and finite order extends into a local Killing field, but also that this order is completely determined by the \textit{model space} $\X$.

We then apply this to recover Gromov's result on the structure of orbits of local isometries for Cartan geometries. We obtain the following formulation.

\begin{thm}[Structure of $\Aut^{loc}$-orbits]
\label{thm:isloc_orbit}
Let $(M,\mathcal{C})$ be a Cartan geometry modeled on $\X = G/P$. Assume the adjoint action of $P$ on $\g$ algebraic. Then, there exists an open dense subset $\Omega$ of $M$ on which the orbits of $\Aut^{loc}(M,\mathcal{C})$ are closed submanifolds. 

More precisely, $\Omega$ admits a decomposition $\Omega = \Omega_1 \cup \cdots \cup \Omega_k$ where each $\Omega_i$ is an open subset, preserved by local automorphisms, and in which the $\Aut^{loc}$-orbits are closed submanifolds which all have the same dimension.
\end{thm}

Our proof shows that in the class of Cartan geometries, the problem is fundamentally reduced to a question on the structure of the orbits of local automorphisms of a global framing on a manifold. This is why a section is devoted to the study of $\Aut^{loc}$-orbits in the geometric structure defined by an absolute parallelism. This question will be treated in an elementary way and we shall avoid the formalism of jets.

\subsection{Organization of the paper}

Section \ref{s:background} contains some generalities about Cartan geometries that will be used in this paper. The reader will find the definition of local and infinitesimal homogeneity in Section \ref{ss:homogeneity}. We then expose in Section \ref{s:singer} a short and direct proof of the generalized Singer's theorem (Theorem \ref{thm:singer}), and explain how to deduce the open dense orbit theorem from Singer's result. Section \ref{s:gromov} focuses mainly on the proofs of Theorems \ref{thm:frobenius} and \ref{thm:isloc_orbit}. As explained above, we have chosen to split their proofs and we first give the proofs in  the case of a parallelism, which is actually the core of the problem, before establishing the general statements. At last, we explain how our proof extends to more general geometric structures and that we can easily deduce the generalized Singer's theorem from Gromov's theorem, with a bound that is however less tight than in Theorem \ref{thm:singer}.

\subsection*{Conventions and notations}

The Lie algebra of a Lie group will be denoted by the same symbol printed in the fraktur font (for instance, $\g$ will design $\Lie(G)$). Often, we will deal with fiber bundles. We have chosen the following notation : the total space of a fibration over a manifold $M$ will be denoted by $\hat{M}$, and objects living in $\hat{M}$ over some object on the base $M$ will be noted with ``hats'' over the corresponding objects of $M$. For example, a point in $\hat{M}$ lying over a point $x \in M$ will be denoted $\hx$ and a bundle automorphism of $\hat{M}$ over a diffeomorphism $f \in \Diff(M)$ will be denoted $\hat{f}$.

In all this paper, $G$ will denote a Lie group, $P<G$ a closed subgroup and we will note $\X = G/P$. We will always make the following assumptions :
\begin{enumerate}
\item $G$ is connected ;
\item $P$ contains no non-trivial normal subgroup of $G$, or equivalently $P$ acts effectively on $G/P$.
\end{enumerate}

We will say that a manifold is ``algebraic'' if it is a smooth quasi-projective real variety. The model space will be said of \textit{algebraic type} when $\Ad_{\g}(P) < \GL(\g)$ is an algebraic subgroup.

\section{Geometric background : Cartan geometries}
\label{s:background}
A Cartan geometry is a geometric structure associated to a certain homogeneous space, called the \textit{model space}. Formally, a Cartan geometry on a manifold, modelled on $\X$, is a geometric structure on that manifold that shares the infinitesimal properties of the homogeneous space.

\begin{dfn}
\label{dfn:cartan}
Let $M$ be a manifold. A \textit{Cartan geometry} on $M$ modeled on $\X$ is the data  of $(M,\hat{M},\omega)$, where $\hat{M} \rightarrow M$ is a $P$-principal fiber bundle and $\omega$ is a $\g$ valued $1$-form on $\hat{M}$, called the \textit{Cartan form} or the \textit{Cartan connection}, satisfying the following conditions 
\begin{enumerate}
\item For all $\hx \in \hat{M}$, $\omega_{\hx} : T_{\hx}\hat{M} \rightarrow \g$ is a linear isomorphism ;
\item For all $X \in \p$, $\omega(X^*) = X$, where $X^*$ denotes the fundamental vector field associated to the right action of $\exp(tX)$ ;
\item For all $p \in P$, $(R_p)^*\omega = \Ad(p^{-1})\omega$.
\end{enumerate}
\end{dfn}

We now present some general facts and definitions about Cartan geometries. Fix a manifold $(M, \mathcal{C})$ endowed with a given Cartan geometry modeled on $\X$. We will always note $\pi : \hat{M} \rightarrow M$ the associated principal fiber bundle with structural group $P$ and $\omega$ its Cartan connection.

Note that the definition implies $\dim M = \dim \X$. The Cartan form defines an absolute parallelism on $\hat{M}$. If $X \in \g$, note $\tilde{X}= \omega^{-1}(X)$ the $\omega$-constant vector field associated to $X$. We have a notion of exponential map in $\hat{M}$ : if $\hx \in \hat{M}$, and $A \in \g$ is such that the flow $\phi^t_{\tilde{A}}$ is defined at time $1$, set $\exp(\hx,A) = \phi^1_{\tilde{A}}(\hx)$. For all $\hx$, this exponential map is defined in a neighbourhood of $0$ in $\g$ and satisfies the following property
\begin{equation*}
\exp(\hx.p,A) = \exp(\hx, \Ad(p)A).p.
\end{equation*}
An \textit{exponential neighbourhood} of a point in $\hx$ is the image by $\exp(\hx,.)$ of a small enough neighbourhood of $0$ in $\g$
 
Simplest examples of manifolds endowed with Cartan geometries modeled on $\X$ are open sets $U \subset \X$. The fibration will be the restriction to $\pi^{-1}(U)$ of the canonical projection $\pi : G \rightarrow \X$ and $\omega$ will be the left Maurer-Cartan form. In this case, $\omega$-constant vector fields are left invariant vector fields. As one shall expect, those examples are exceptional among all the possible Cartan geometries modeled on $\X$. In fact, up to automorphism, they are locally characterized by the vanishing of a $2$-form on $\hat{M}$, called the curvature form. Before detailing this last assertion, let us introduce the \textit{equivalence problem} on some fundamental examples.
\begin{enumerate}
\item Let $\X = (O(n)\ltimes \R^n)/O(n)$ be the Euclidean space. To each Riemannian manifold $(M^n,g)$ naturally corresponds a \textit{torsion free} (see \cite{sharpe}, p.184) Cartan geometry $(M^n,\mathcal{C})$ modeled on $\X$. Conversely, to each torsion free Cartan geometry on $M$ modeled on $\X$ canonically corresponds a Riemannian metric $g$ on $M$. Local automorphisms of the Cartan geometry (see bellow) are exactly local isometries of the associated metric. If we replace $O(n)$ by $O(p,q)$, we obtain pseudo-Riemannian metrics.
\item If $n \geq 3$, let $\X = O(1,n+1)/P$ be the conformal round sphere, where $P<O(1,n+1)$ is the stabilizer of an isotropic line in the $(n+2)$-dimensional Minkowski space. Then, there is also a $1-1$ correspondence between Riemannian conformal structures on $M^n$ and \textit{normal} Cartan geometries on $M^n$ modeled on $\X$ (see \cite{sharpe}, p.280). Local automorphisms of such a Cartan geometry are exactly the local conformal diffeomorphisms of the associated conformal class of metrics. If we replace the round sphere by its pseudo-Riemannian generalization, namely the Einstein space $\Ein^{p,q}$, Cartan geometries modeled on $\Ein^{p,q}$ correspond to conformal classes of metrics of signature $(p,q)$.
\item Let $\X = (\GL_n(\R)\ltimes \R^n)/\GL_n(\R)$ be the affine space. Then, there exists a $1-1$ correspondence between manifolds  $(M^n,\nabla)$ endowed with a linear connection and Cartan geometries $(M,\mathcal{C})$ modeled on $\X$. Local automorphisms of $(M,\mathcal{C})$ are exactly local affine transformations of $(M,\nabla)$.
\end{enumerate}

This list is not exhaustive. The procedure that consists in finding a natural ``classic'' geometric structure associated to Cartan geometries of a certain type is called the \textit{equivalence problem}. It has been solved in other cases than those cited above (for instance projective structures or non-degenerate CR-structures, see \cite{capschichl}) and then, Cartan geometries \textit{modelize} a large class of geometric structures.

\subsection{Curvature of a Cartan geometry}

In the general case, let us introduce the $2$-form $\Omega = \d \omega + \frac{1}{2}[\omega,\omega]$ on $\hat{M}$, called the \textit{curvature form}. The geometry will be said \textit{flat} if $\Omega = 0$. Since in any Lie group, the left Maurer-Cartan form satisfies the structural equation, it is clear that any Cartan geometry that is locally isomorphic (see Definition \ref{dfn:local_automorphism} below) to its model must be flat. Conversely, it turns out that any flat Cartan geometry must be locally isomorphic to its model (see \cite{sharpe}, theorem 5.5.1). It can be shown that such flat Cartan geometries on $M$ are exactly the data of a $(G,\X)$-structure on $M$ (see for instance \cite{cap}, Remark 1.5.2).

The Cartan connection trivializes $T\hat{M}$, and then the curvature $2$-form gives rise to a map $K : \hat{M} \rightarrow \Hom(\Lambda^2\g,\g)$, called the \textit{curvature map}. Since $\Omega$ is horizontal, for all $\hx$, $K(\hx)$ factorizes in $\Hom(\Lambda^2(\g/\p),\g)$. The equivariance property of $\omega$ implies that $K$ is $P$-equivariant, when $P$ acts on the right on $\Lambda^2(\g / \p)^* \otimes \g$ via the tensor product of the representations $\Lambda^2 \bar{\Ad}^*$ and $\Ad$. At last, $K$ satisfies the relation $K(\hx)(X,Y) = [X,Y] - \omega_{\hx}([\tilde{X},\tilde{Y}])$ (\cite{sharpe}, p.192).

\subsection{Isomorphisms and Killing fields}

Our hypothesis on the model space $\X$ implies the following lemma, fundamental in the study of local automorphisms of the Cartan geometry. The result is well known and a proof can be found in \cite{melnick}, Proposition 3.6.

\begin{lem}
Let $U \subset M$ be an open set and $F$ be a bundle automorphism of $\pi^{-1}(U)$ over the identity such that $F^* \omega = \omega$. Then, $F=\id$.
\end{lem}

If we are given two Cartan geometries $(M_i,\hat{M}_i,\omega_i)$, $i=1,2$, on the same model $\X$, an \textit{isomorphism} between them is a diffeomorphism $f: M_1 \rightarrow M_2$ that lifts to a bundle isomorphism $\hat{f}$ such that ${\hat{f}}^* \omega_2 = \omega_1$. Such a lift must be unique according to the previous lemma. \textit{Automorphisms} of a Cartan geometry $(M,\mathcal{C})$ modeled on $\X$ are isomorphisms into itself. They form a finite dimensional Lie group $\Aut(M,\mathcal{C})$. This group acts on $\hat{M}$, and its action is free and proper, since lifts of automorphisms preserve an absolute parallelism on $\hat{M}$.

\begin{dfn}[Local automorphisms and Killing fields]
\label{dfn:local_automorphism}
Let $U,V \subset M$ be two open sets. A \textit{local isomorphism} between $U$ and $V$ is a diffeomorphism $f : U \rightarrow V$ that lifts to a bundle automorphism $\hat{f} : \pi^{-1}(U) \rightarrow \pi^{-1}(V)$ preserving the Cartan connection.

A \textit{Killing field (resp. local Killing field)} is a vector field on $M$ (resp. an open subset $U \subset M$) whose local flow is composed with local automorphisms.
\end{dfn}

Let $X$ be a local Killing field on an open set $U$. Then, $X$ lifts to a field $\hat{X}$ on $\pi^{-1}(U)$ whose flow $\phi^t_{\hat{X}}=\hat{\phi^t_X}$ is the lift to $\pi^{-1}(U)$ of the flow of $X$. Automatically, $\hat{X}$ commutes with the right action of the structural group $P$ and satisfies $\L_{\hat{X}} \omega = 0$. Conversely, if a field on $\pi^{-1}(U)$ satisfies this properties, then its projection on $U$ is a local Killing field.

If $f : U \rightarrow V$ is a local automorphism between two connected open sets, then $f$ is completely determined by the image $\hat{f}(\hx)$ of an arbitrary point $\hx \in \pi^{-1}(U)$ (recall $f$ preserves a global frame). This implies that if $\hat{X}$ is the lift to $\pi^{-1}(U)$ of a local Killing field $X$ defined on an open set $U$, then $\hat{X}$ (and \textit{a fortiori} $X$) is uniquely determined by its value at any point $\hx \in \pi^{-1}(U)$. 

Note $\Kill^{loc}(x)$ the set of germs of local Killing fields at $x \in M$. They form a vector space, and if $\hx$ lies over $x$, the evaluation $X \in \Kill^{loc}(x) \mapsto \hat{X}(\hx)$ is a linear injective map, and this proves that $\Kill^{loc}(x)$ is a finite dimensional vector space. In fact, $\Kill^{loc}(x)$ is more than an abstract set of germs of vector fields. Consider a decreasing sequence of connected open sets $U_i$ such that $\bigcap_{i \geq 0} U_i = \{x\}$ and note $\Kill_i(x)$ the vector space of local Killing fields at $x$ defined on $U_i$. Since $\Kill_i(x) \subset \Kill_{i+1}(x)$ and $\Kill_i(x) \hookrightarrow \Kill^{loc}(x)$ for all $i \geq 0$, there exists $i_0$ such that $\Kill^{loc}(x) \simeq \Kill_{i_0}(x)$, \textit{ie} every local Killing field at $x$ is defined at least on a common neighbourhood $U_{i_0}$ of $x$. 

\begin{prop}
\label{prop:extension}
Let $F : \mathcal{U} \rightarrow \mathcal{V}$ be a diffeomorphism between two open sets of $\hat{M}$ that preserves the Cartan form. Assume that $\mathcal{U}$ is connected and that for all $x \in \pi(\mathcal{U})$, $\mathcal{U} \cap \pi^{-1}(x)$ is connected. Set $U=\pi(\mathcal{U})$ and $V=\pi(\mathcal{V})$. Then $F$ extends to a bundle isomorphism $\hat{f} : \pi^{-1}(U) \rightarrow \pi^{-1}(V)$ that is the lift of a local isomorphism $f$ between $U$ and $V$.

Let $\mathcal{U} \subset \hat{M}$ be a connected open set such that for all $x \in \pi(\mathcal{U})$, $\mathcal{U} \cap \pi^{-1}(x)$ is connected. Let $\hat{X}$ be a field on $\mathcal{U}$ such that $\L_{\hat{X}}\omega = 0$. Then $\hat{X}$ is in fact the restriction to $\mathcal{U}$ of a local Killing field acting on the saturated open set $\pi^{-1}(\pi(\mathcal{U}))$.
\end{prop}

\begin{proof}
There is no choice to build the extensions of $F$ and $X$ since they have to commute with the action of the structural group. The only condition is that this extension is well-defined. We explain it in the case of a vector filed. Such a field $\hat{X}$ has to commute with the fundamental vector fields $A^*$, $A \in \p$ (Definition \ref{dfn:cartan}). Therefore, if two points $\hx \in \mathcal{U}$ and $\hx.p \in \mathcal{U}$ lie in the same fiber, then $\hat{X}_{\hx.p} = (R_p)_* \hat{X}_{\hx}$. Thus, we can define an extension of $\hat{X}$ to the saturated open set $\mathcal{U}.P$. One easily verifies that this extension is indeed a Killing field.
\end{proof}

If we want to exhibit a local Killing field in the neighbourhood of a point $x \in M$, we only have to build a field on a convex exponential neighbourhood of some $\hx \in \hat{M}$ lying over $x$ that commutes with every $\omega$-constant vector field, that is a local Killing field of a framing on this neighbourhood.

\subsection{Orbits of local automorphisms, local and infinitesimal homogeneity}
\label{ss:homogeneity}
The $\Aut^{loc}$\textit{-orbit} of a point $x$ in a manifold endowed with a Cartan geometry is the set of points $y$ such that there exists a local automorphism sending $x$ to $y$. The $\Kill^{loc}$\textit{-orbit} of a point $x \in M$ is the set of points of $M$ that can be reached from $x$ by a finite sequence of flows of local Killing fields. Obviously, any $\Kill^{loc}$-orbit is contained in an $\Aut^{loc}$-orbit.

\begin{dfn}[Local homogeneity]
A Cartan geometry on a manifold $M$ is said to be locally homogeneous if it contains only one $\Aut^{loc}$-orbit.
\end{dfn}

If $\hat{f}$ is the lift of a local automorphism $f : U \rightarrow V$, we have $(\hat{f})^*\Omega = \Omega$, and then $K \circ \hat{f} = K$ on  $\pi^{-1}(U)$. Therefore, we see that if $\mathcal{O} \subset M$ is an $\Aut^{loc}$-orbit, $K(\pi^{-1}(\mathcal{O}))$ is a single $P$-orbit in $\Hom(\Lambda^2(\g/\p),\g)$.

\subsubsection{Universal covariant derivative}

\begin{dfn}[Universal covariant derivative, \cite{sharpe}, p.194]
\label{dfn:covariant_derivative}
Let $V$ be a vector space and $\rho : P \rightarrow \GL(V)$ a linear representation. Let $F : \hat{M} \rightarrow V$ be a $P$-equivariant map.  Since we are given an absolute parallelism on $\hat{M}$, the differential $T F : T\hat{M} \rightarrow TV$ can be interpreted as an equivariant  map $\mathcal{D}^1 F : \hat{M} \rightarrow V \oplus \Hom(\g,V)$ setting 
\begin{equation*}
\forall	\hx \in \hat{M}, \ \mathcal{D}^1 F (\hx) = (F(\hx),D^1 F(\hx)) \text{, where }  \forall A \in \g, \ D^1 F (\hat{x})(A) = T_{\hx} F (\tilde{A}_{\hx}).
\end{equation*}
\end{dfn}

The tangent map $TF$ is of course completely determined by this equivariant map $\mathcal{D}F$. A very nice aspect of this point of view is that the map $F$ and its differential are of the same nature. Thus, we can iterate the process and get $\mathcal{D}^rF$ the $r$-th derivative of $F$, defined after some identifications by
\begin{equation*}
\begin{array}{cccl}
\mathcal{D}^r F : & \hat{M} & \rightarrow & V \oplus \Hom(\g,V) \oplus \cdots \oplus \Hom(\otimes^r \g,V) \\
                     & \hx & \mapsto & (F(\hx), D^1F(\hx), \ldots , D^r F(\hx)),
\end{array}
\end{equation*}
where $D^r F(\hx)(X_1 \otimes \cdots \otimes X_r):=(\tilde{X_1}.\ldots.\tilde{X_r}.F)(\hx)$. It is not difficult to prove that the $r$-th derivative of $F$ is $P$-equivariant, where $P$ acts on the right on each factor $\Hom(\otimes^r \g, V)$ in the following way : for all $w \in \Hom(\otimes^r \g, V)$ and $p \in P$,
\begin{equation*}
w.p : X_1 \otimes \cdots \otimes X_r \mapsto \rho(p).w (\Ad(p)X_1,\ldots, \Ad(p)X_r).
\end{equation*}
At last, if $f : U \rightarrow V$ is a local automorphism of $(M,\mathcal{C})$ such that $F \circ \hat{f} = F$ on $\pi^{-1}(U)$, then for all $r \geq 0$, we have $\mathcal{D}^r F \circ \hat{f} = \mathcal{D}^r F$ on $\pi^{-1}(U)$.

\subsubsection{Infinitesimal homogeneity}

Let $V = \Hom(\Lambda^2(\g/\p),\g)$ be the target space of the curvature map $K$. Note $\mathcal{W}^r:=V \oplus \Hom(\g,V) \cdots \oplus \Hom(\otimes^r \g, V)$. Since any local automorphism $f$ satisfies $K \circ \hat{f} = K$, if $\mathcal{O}$ is an $\Aut^{loc}$-orbit, then for all $r \geq 0$, $\mathcal{D}^r K (\pi^{-1}(\mathcal{O}))$ is a single $P$-orbit in $\mathcal{W}^r$. This motivates the last definition.

\begin{dfn}[Infinitesimal homogeneity]
\label{dfn:infinitesimal_homogeneity}
The geometry of $M$ is said $r$-infinitesimally homogeneous, for $r \geq 1$, if the $P$-equivariant map $\mathcal{D}^r K$ takes values in a single $P$-orbit in $\mathcal{W}^r$.
\end{dfn}

\section{Singer's theorem for Cartan geometries}
\label{s:singer}
Let $(M,\mathcal{C})$ be a Cartan geometry modeled on $\X = G/P$, with $M$ connected. Under the hypotheses of generalized Singer's theorem, subbundles with interesting properties appear, and infinitesimal homogeneity of $(M,\mathcal{C})$ is derived from their study. We have isolated the essential properties of this object, and we now present them. Theorem \ref{thm:singer} will follow easily.

\subsection{Parallel submanifolds in Cartan geometries}

\begin{dfn}
A submanifold $\Sigma \subset \hat{M}$ is said to be \textit{parallel} if there exists a vector subspace $\h \subset \g$ such that for each $\hx \in \Sigma$, $\omega(T_{\hx}\Sigma)=\h$.
\end{dfn}

Parallel submanifolds appear in the following context.

\begin{lem}
\label{lem:level_set}
Let $F : \hat{M} \rightarrow V$ be a map taking values in a vector space. Assume $F$ has constant rank and note $\Sigma = F^{-1}(v_0)$ for $v_0 \in V$. If $\mathcal{D}^1F : \hat{M} \rightarrow V \oplus \Hom(\g,V)$ is constant over $\Sigma$, then $\Sigma$ is parallel.
\end{lem}

\begin{proof}
Remind that for us, $\mathcal{D}^1F(\hx) = (F(\hx),D^1F(\hx))$, where $D^1F(\hx)(A) = T_{\hx}F(\tilde{A}_{\hx})$. Here, $D^1F$ is constant equal to $w_0 \in \Hom(\g,V)$ all over $\Sigma$. Since for all $\hx \in \Sigma$, $T_{\hx}\Sigma = \Ker T_{\hx} F$, we have
\begin{align*}
\omega(T_{\hx}\hat{\Sigma}) & = \{A \in \g \ :\ (\L_{\tilde{A}}. F) (\hx) = 0 \} \\
                       & = \{A \in \g \ : \ D^1F(\hx)(A) = 0\} \\
                       & = \Ker w_0.
\end{align*}
Hence, $\Sigma$ is parallel.
\end{proof}

If $\Sigma$ is a parallel submanifold, it is natural to consider the restrictions to $\Sigma$ of $\omega$-constant fields $\tilde{X}$, for $X \in \h$. Those vector fields are tangent to $\Sigma$ and we name them $(\h,\omega)$-constant vector fields. Those fields define an absolute parallelism on $\Sigma$ noted $\mathcal{P}^{\Sigma}$.

If the curvature map of $(M,\mathcal{C})$ is moreover constant over a parallel submanifold, then we are almost able to prove the local homogeneity of $(M,\mathcal{C})$. This is the content of the following proposition.

\begin{prop}[Parallel submanifolds with constant curvature]
\label{prop:constant_curvature}
Let $\Sigma$ be a connected parallel submanifold of $\hat{M}$ with ``space'' $\h$.
\begin{enumerate}
\item If the curvature map $K$ is constant over $\Sigma$, then $(\Sigma,\mathcal{P}^{\Sigma})$ is locally homogeneous.
\item If moreover, $\h$ and $\p$ intersect transversally, then $\pi|_{\Sigma} : \Sigma \rightarrow M$ is a submersion and the connected components of $\pi(\Sigma)$ are open $\Kill^{loc}$-orbits. In particular, if $\pi(\Sigma) = M$, then $(M,\mathcal{C})$ is locally homogeneous.
\end{enumerate}
\end{prop}

\begin{proof}
1. Take $\Sigma$ a parallel submanifold of $\hat{M}$ such that $\omega(T_{\hx}\Sigma)=\h$. Assume $K(\hx) = K_0 \in \Hom(\Lambda^2(\g/\p),\g)$ all over $\Sigma$. The first step is to prove that $(\h,\omega)$-constant vector fields on $\Sigma$ form a (finite dimensional) Lie subalgebra of $\Gamma(T\Sigma)$. This fact comes from the general relation $K_{\hx}(X,Y) = [X,Y] - \omega_{\hx}([\tilde{X},\tilde{Y}])$ holding for every pair of $\omega$-constant vector fields. Indeed, if $X,Y \in \h$, since $K\equiv K_0$ on $\Sigma$, we must have for $\hx \in \Sigma$
\begin{equation*}
\omega_{\hx}([\tilde	{X},\tilde{Y}]) = [X,Y] - K_0(X,Y),
\end{equation*}
this equality meaning that $[\tilde{X},\tilde{Y}]$ is $\omega$-constant on $\Sigma$. Moreover, because $[\tilde{X},\tilde{Y}]$ is tangent to $\Sigma$, we get $\omega_{\hx}([\tilde	{X},\tilde{Y}]) \in \h$, meaning $[\tilde{X},\tilde{Y}]$ is itself $(\h,\omega)$-constant.

This observation encourages us to consider $[,]'$ a Lie bracket on $\h$ defined by $[X,Y]' = [X,Y]-K_0(X,Y)$. It naturally corresponds to that of $(\h,\omega)$-constant vector fields on $\Sigma$, and hence satisfies the axioms of brackets. We note $\h'$ for $(\h,[,]')$. Thus, we have exhibited an infinitesimal Lie group action $\iota : \h' \hookrightarrow \Gamma(T\Sigma)$. According to Palais' work on infinitesimal actions (see \cite{palais}, theorem XI, p. 58), if $H'$ is any Lie group with Lie algebra isomorphic to $\h'$, one has a local action of $H'$ on the manifold $\Sigma$ integrating the infinitesimal action $\iota$. This means that for all $\hx \in \Sigma$ and $X \in \Lie(H')$, $\ddt{t=0} \exp(tX).\hx = - \iota(X)_{\hx}$.

It is enough to prove that any point of $\Sigma$ admits a locally homogeneous neighbourhood. Take a point $\hx_0 \in \Sigma$. Because a non-zero $\omega$-constant vector field never vanishes, the local $H'$-action is everywhere locally free on $\Sigma$. Hence, there exists an open neighbourhood $U$ of $\hx_0$ in $\Sigma$ and an open neighbourhood $V$ of $e$ in $H'$ such that the orbital application at $\hx_0$ gives us an identification $\psi : V  \rightarrow U$. Assume $U$ and $V$ connected. Now, the point is that under this identification, $(\h,\omega)$-constant vector fields on $U$ correspond to right-invariant vector fields on $V$ (this is a simple consequence of the definition of the local $H'$-action). We can formulate this fact more formally saying that if $\omega^{H'}$ denotes the right Maurer-Cartan form on $H'$, we have $\psi^* (\omega|_{TU}) = \omega^{H'}|_{TV}$.

Finally, for all $X \in \h'$, note $\hat{X} = \psi_*( ^{L}\!X)$ the push-forward of the left-invariant vector field corresponding to $X$ on $V$. Then, because left-invariant and right-invariant vector fields commute, we see that for all $X$, $\hat{X}$ commute with every $(\h,\omega)$-constant vector fields on $U$. Thus, $\{\hat{X}, \ X \in \h' \}$ is a Lie algebra of Killing fields of $(U,\mathcal{P}^{\Sigma})$. Since they span $T_{\hx}\Sigma$ everywhere, any two points in $U$ can be joined by a finite sequence of flows of $\hat{X}$'s. Thus, $U$ is locally homogeneous and we have proved the first part of the proposition.

2. Now, assume moreover that $\h$ and $\p$ intersect transversally. This implies that when restricted to $\Sigma$, $\pi$ realizes a submersion since $T_{\hx}\pi(\omega_{\hx}^{-1}(\h)) = T_{\hx}\pi (\omega_{\hx}^{-1}(\h + \p)) = T_x M $ . Let $\s = \h \cap \p$. If $X_1,X_2 \in \s$, since the curvature form is horizontal, one has $[X_1,X_2] = [X_1, X_2]' \in \h$, proving $\s$ is a Lie subalgebra of $\p$. Choose $S<P$ a connected Lie subgroup of $P$ whose Lie algebra is isomorphic to $\s$. Note that because $(\s,\omega)$-constant fields generate flows which correspond to the right action on $\hat{M}$ of elements of $S$, $\Sigma$ is preserved by the right $S$-action. 

Let $\hx_0$ be a point in $\Sigma$ and $U$ its locally homogeneous neighbourhood exhibited previously. Fix $\t$ an arbitrary complement of $\s$ in $\h$ and choose $0 \in \mathcal{U} \subset \t$ a small neighbourhood of the origin such that $\exp_{\hx} : \mathcal{U} \rightarrow \Sigma$ is a diffeomorphism onto its image and $\exp_{\hx}(\mathcal{U})=:U' \subset U$. If $\mathcal{U}$ is small enough, there exists $V_S$ a convex neighbourhood of the identity in $S$ such that $U_1 = U'.V_S \subset U$. Since the fields $\hat{X}$ commute with $(\s,\omega)$-constant vectors fields, if two points $\hx \in U_1$ and $\hx.s \in U_1$ lie in the same fiber, then $\hat{X}_{\hx.s} = (R_s)_* \hat{X}_{\hx}$. Therefore, there is a well defined extension of the $\hat{X}$'s to the saturated open set $\pi^{-1}(\pi(U_1))$ by setting $\hat{X}_{\hx.p} = (R_p)_* \hat{X}_{\hx}$. We are left to prove that the extension of the $\hat{X}$'s to the saturated open set are in fact the lifts of local Killing fields of $(M,\mathcal{C})$.

We have seen that by construction, for all $\hx \in U$, $(\L_{\hat{X}}\omega)_{\hx}(v) = 0$ for all $v \in T_{\hx}\Sigma$. We claim that $\hat{X}$ satisfies 
\begin{enumerate}
\item for all $\hx \in \pi^{-1}(\pi(U_1))$, $(\L_{\hat{X}}\omega)_{\hx}$ vanishes on all vertical directions ;
\item the following equivariance relation for all $\hx \in \pi^{-1}(\pi(U_1))$ and $A \in \g$
\begin{equation*}
\left (\L_{\hat{X}}\omega\right )_{\hx.p}(\tilde{A)} = \Ad(p^{-1}) \left (\L_{\hat{X}}\omega \right )_{\hx}(\tilde{\Ad(p)A})=0.
\end{equation*}
\end{enumerate}
Those facts are actually  true for every vector field that commutes with the $P$-action. It comes from the following computations :

\begin{enumerate}

\item Since the curvature $2$-form is horizontal, we have for all $\hat{x} \in \pi^{-1}(\pi(U_1)) \text{ and } A \in \p, \ \Omega_{\hx}(\hat{X}_{\hx},\tilde{A}_{\hx}) = \d \omega_{\hx}(\hat{X}_{\hx},\tilde{A}_{\hx}) + [\omega(\hat{X}_{\hx}),A]=0$. Take any $\hx \in \pi^{-1}(\pi(U_1))$. We compute
\begin{align*}
(\L_{\hat{X}} \omega)_{\hx}(\tilde{A}) & = \d \omega_{\hx}(\hat{X},\tilde{A}) + (\tilde{A}.\omega(\hat{X}))_{\hx} \\ & = [A,\omega_{\hx}(\hat{X})] + \ddt{t=0} \omega(\hat{X})(\hx.\e^{tA})  \\
                                       & = [A,\omega_{\hx}(\hat{X})] + \ddt{t=0} ({R_{\e^{tA}}}^{\!\!*}\omega)_{\hx} (\hat{X}) \\ & = [A,\omega_{\hx}(\hat{X})] + \ddt{t=0} \Ad(\e^{-tA})\omega_{\hx}(\hat{X}) \\
                                       & = 0,
\end{align*}

\item For $A \in \g$ and small enough $t$
\begin{align*}
\left [(\phi_{\hat{X}}^t)^*\omega\right ]_{\hx.p}(\tilde{A}) & = \left [(\phi_{\hat{X}}^t)^*\omega\right ]_{\hx.p}\left (\left (R_p\right )_* \tilde{\Ad(p)A}_{\hx}\right ) \\
                                                             & = \left [ \left ( R_p \right )^* \left (\phi_{\hat{X}}^t \right )^*\omega \right ]_{\hx} \left (\tilde{\Ad(p)A}_{\hx}\right ) \\
                                                             & = \left [\left (\phi_{\hat{X}}^t\right )^*\left (R_p\right )^* \omega\right ]_{\hx}\left (\tilde{\Ad(p)A}_{\hx}\right ) \\
                                                             & = \Ad(p^{-1}) \left (\left (\phi_{\hat{X}}^t\right )^*\omega\right )_{\hx}\left (\tilde{\Ad(p)A}_{\hx}\right )
\end{align*}
Taking derivative at $t=0$, we get
\begin{equation*}
\left (\L_{\hat{X}}\omega\right )_{\hx.p}(\tilde{A)} = \Ad(p^{-1}) \left (\L_{\hat{X}}\omega \right )_{\hx}(\tilde{\Ad(p)A})=0.
\end{equation*}
\end{enumerate}

Finally, the fields $\hat{X}$, $X \in \h$ satisfy $\L_{\hat{X}} \omega = 0$ on $\pi^{-1}(\pi(U_1))$. According to Proposition \ref{prop:extension}, they are the lifts of Killing vector fields defined on a neighbourhood of $x_0$ in $M$. Thus, we have exhibited a Lie algebra $\h' \subset \Kill^{loc}(x_0)$ of local Killing fields defined on an open neighbourhood of $x_0$. Since $\pi|_{\Sigma}$ is a submersion, we have $\h'_x = \{X(x),X\in \h'\}=T_xM$ for all $x$ in this neighbourhood. Therefore, any connected component of $\pi(\Sigma)$ is locally homogeneous.
\end{proof}

\subsection{Proof of the generalized Singer's theorem}

We now prove Theorem \ref{thm:singer}. Note $d = \dim \p$ and assume that $(M,\mathcal{C})$ is $d$-infinitesimally homogeneous. Then the maps $\mathcal{D}^i K$ take values in a single orbit in $\mathcal{W}^i$, for $i \leq d$. Let us fix an arbitrary point $\hx_0$ in $\hat{M}$. For $0 \leq i \leq d$, let $S_i = \Stab_P(\mathcal{D}^iK(\hx_0))$ and $\s_i = \Lie(S_i)$. There is an integer $k(\hx_0) \leq d$ such that $\s _0 \varsupsetneq \s_1 \varsupsetneq \cdots \varsupsetneq \s_{k(\hx_0)} = \s_{k(\hx_0)+1}$. Remark that if $\hx$ is another point, for all $i \leq d$, $\Stab_P(\mathcal{D}^i K(\hx))$ and $\Stab_P(\mathcal{D}^iK(\hx_0))$ are conjugated in $P$. This implies $k(\hx)=k(\hx_0)=k$. We call $k$ \textit{the Singer integer of} $(M,\mathcal{C})$. Recall the following well-known result :

\begin{lem}
Let $(P,\pi,M,G)$ be a principal fibre bundle over a connected manifold $M$ with group $G$ and $H<G$ a closed subgroup. Then $P$ admits a reduction to an $H$-subbundle if and only if there exists a $G$-equivariant map $\Phi : P \rightarrow G/H$. In that case, for all $u \in P$, $\Phi^{-1}(\Phi(u))$ is a principal subbundle of $P$, whose structural group is conjugated to $H$.
\end{lem}

For each $i$, let $\hat{N_i}$ be the connected component of the level set $\mathcal{D}^i K ^{-1} (\mathcal{D}^i K(\hx_0))$ containing $\hx_0$. Then $\hat{N_i}$ is a principal subbundle of $\hat{M}$ with group $(S_i)_0$, the identity component of $S_i$. Because of the definition of $k$, we must have $\hat{N_k} = \hat{N_{k+1}}$. Note $\hat{N}$ this subbundle and $S$ its group.

The definition of Singer's integer ensures that $\mathcal{D}^k K$ satisfies the hypotheses of Lemma \ref{lem:level_set}. Indeed, if $F = \mathcal{D}^k K$, then $\mathcal{D}^1 F = (K,\ldots,D^k K, D^1 K, \ldots, D^{k+1}K)$ must be constant over $\hat{N}$ since $\mathcal{D}^{k+1} K$ is itself constant over this subbundle. Thus, $\hat{N}$ is a parallel submanifold, which is also a principal subbundle of $\hat{M}$. Moreover, $K$ is constant over $\hat{N}$. Therefore, Proposition \ref{prop:constant_curvature} implies that $M$ is locally homogeneous.

\subsection{Proof of the dense orbit theorem}
Theorem \ref{thm:singer} gives us a direct proof of the dense orbit theorem for Cartan geometries, which we will now explain.

Let $(M,\mathcal{C})$ be a Cartan geometry, modeled on a space $\X = G/P$ of algebraic type, and such that the pseudo-group $\Aut^{loc}(M)$ admits a dense orbit $\mathcal{O}$ in $M$. Let $d = \dim P$. Let $\Phi = \mathcal{D}^{d} K$ be the $d$-th covariant derivative of the curvature. Since $\Aut^{loc}(M)$ acts on $\hat{M}$ preserving every derivative of the curvature, $\mathcal{O}'=\Phi(\pi^{-1}(\mathcal{O}))$ must be a single $P$-orbit in $\mathcal{W}^d$. Let $F$ be the closure of $\mathcal{O}'$ in $\mathcal{W}^d$. Then, the map $\Phi$ takes values in $F$. Recall that $\Ad_{\g}(P) \subset \GL(\g)$ is an algebraic subgroup. Since the $P$-action on $\mathcal{W}^d$ is algebraic, every $P$-orbit in $\mathcal{W}^d$ is open in its closure (the closure is taken relatively to the Hausdorff topology of $\mathcal{W}^d$). Thus, $\mathcal{O}'$ is open in $F$ and the open set $\Phi^{-1}(\mathcal{O}')$ projects on an open subset of $M$ containing $\mathcal{O}$, and thus has to be dense itself. Call it $U$. Therefore, the restricted Cartan geometry $(U,\pi^{-1}(U),\omega)$ satisfies the hypotheses of Theorem \ref{thm:singer}. Therefore, each connected component of $U$ is locally homogeneous. Take $x,y \in U$ and $V_x, V_y \subset U$ connected neighbourhoods of $x$ and $y$ respectively. These neighbourhoods are locally homogeneous. Since $\mathcal{O}$ is dense, it meets $V_x$ and $V_y$. Therefore, there exists a local automorphism sending $x$ to $y$.

\section{Gromov's Theorem on orbits of local isometries}
\label{s:gromov}
Our presentation of Gromov's result can be seen as a refinement of Singer's theorem. Indeed, in Singer's theorem we assumed that the maps $\mathcal{D}^r K : \hat{M} \rightarrow \mathcal{W}^r$ take values in a single $P$-orbit and concluded that there is only one $\Aut^{loc}$-orbit. In the general case, the principle will be to derive properties of $\Aut^{loc}$-orbits from the study of preimages $\mathcal{D}^rK^{-1}(P.w)$ of $P$-orbits in $\mathcal{W}^r$, for $r \leq \dim G$.

Let $m = \dim G$ and $\Phi = \mathcal{D}^m K$, the $m$-th covariant derivative of the curvature. The general idea used to prove Theorem \ref{thm:isloc_orbit} is that there exists an open dense subset $\Omega \subset M$, that splits into a union of $\Aut^{loc}$-invariant open sets $\Omega = \Omega_1 \cup \cdots \cup \Omega_k$ such that for all $i$, the restriction $\Phi : \pi^{-1}(\Omega_i) \rightarrow \mathcal{W}^r$ has constant rank and every level set of $\Phi$ in $\pi^{-1}(\Omega_i)$ is projected by $\pi$ onto an $\Aut^{loc}$-orbit in $\Omega_i$.

In section \ref{ss:parallelism}, we focus on the elementary case of a Cartan geometry associated to an absolute parallelism on $M$. In this situation, the proof of Theorem \ref{thm:isloc_orbit} will be easier once we will have exhibited $\Omega$ and established Theorem \ref{thm:frobenius} : orbits of local automorphisms will essentially be level sets of a map with constant rank. Nonetheless, it will be helpful to treat this case since in the general case, that will be exposed in section \ref{ss:generalcase}, the idea is to deduce regularity properties of the orbits from the fact that they are the projections of the $\Aut^{loc}$-orbits of a parallelism on $\pi^{-1}(\Omega)$.

\subsection{Killing generators and integrability domain}
\label{ss:gromov_definition}
Let $(M,\mathcal{C})$ be a manifold endowed with a Cartan geometry, with model space $\X=G/P$ satisfying the standard hypotheses. Recall that $V = \Hom(\Lambda^2(\g/\p),\g)$ denotes the target space of the curvature map $K$, and that $\mathcal{D}^r K : \hat{M} \rightarrow \mathcal{W}^r$ denotes the $r$-th covariant derivative of the curvature map.

\begin{dfn}[Killing generator]
For any $A \in \g$, we define a linear morphism
\begin{equation*}
\lefthalfcup A : \Hom(\otimes^r \g, V) \rightarrow \Hom(\otimes^{r-1} \g, V),
\end{equation*}
given by $\forall W \in \Hom(\otimes^r \g, V), \ \forall X_1, \ldots,X_{r-1} \in \g$,
\begin{equation*}
(W \! \lefthalfcup A) (X_1,\ldots,X_{r-1}) = W(A,X_1,\ldots,X_{r-1}).
\end{equation*}
We extend $\lefthalfcup A$ to a linear map $\mathcal{W}^r \rightarrow \mathcal{W}^{r-1}$ defined by
\begin{equation*}
(W_0,\ldots ,W_r) \lefthalfcup A = (W_1 \lefthalfcup A, \ldots , W_r \lefthalfcup A).
\end{equation*}
For $r \geq 1$, $A$ will be a \textit{Killing generator of order} $r$ at $\hx$ if 
\begin{equation*}
\mathcal{D}^r K(\hx) \! \lefthalfcup A = 0.
\end{equation*}
We note $\Kill^r(\hx) \subset \g$ the set of all Killing generators of order $r$ at $\hx$ and $\Kill^{\infty}(\hx) = \bigcap_{r=1}^{\infty} \Kill^r(\hx)$.
\end{dfn}

\begin{rem}
If $X$ is a local Killing field of $(M,\mathcal{C})$ and if $\hat{X}$ is its lift, then for all $x$ where $X$ is defined and for all $\hx \in \pi^{-1}(x)$, we have $\hat{X}(\hx) \in \Kill^{\infty}(\hx)$.
\end{rem}

For all $ 1 \leq r \leq \infty$, $\Kill^r(\hx) = \omega_{\hx}(\Ker T_{\hx} (\mathcal{D}^{r-1} K))$ and for all $p \in P$, we have $\Kill^r(\hx.p) = \Ad(p^{-1})\Kill^r(\hx)$. Therefore, if $x \in M$, the dimension $\dim \Kill^r$ is constant over the fiber $\pi^{-1}(x)$.

\begin{dfn}
If $1 \leq r \leq \infty$, note $k_r(x) = \dim \Kill^r(\hx)$ for any $\hx$ lying over $x$. We simply note $k(x) = k_{\infty}(x)$.
\end{dfn}

\begin{lem}
For $r \geq 1$, the maps $k_r,k : M \rightarrow \N$ are upper semi-continuous.
\end{lem}

\begin{proof}
For $r \geq 1$, $\Kill^r(\hx)$ is the kernel of a linear map depending continuously on $\hx$. Therefore $k_r$ is upper semi-continuous. The same is true for $k$ since $(k_r)$ decreases and converges pointwise to $k$.
\end{proof}

\begin{dfn}[Integrability domain]
We define the \textit{integrability domain} of $(M,\mathcal{C})$ to be the subset $\Int(M,\mathcal{C}) = \{x \in M \ | \ k_1, \ldots , k_{m+2} \text{ are locally constant at } x\}$. If there is no ambiguity on which geometry we are considering, we will also note $M^{int}$ for $\Int(M,\mathcal{C})$.
\end{dfn}

Since the maps $k_r$ are upper semi-continuous, the integrability domain is an open dense subset of $M$. Moreover, since for all $f \in \Aut^{loc}(M,\mathcal{C})$ and $r \geq 1$, $k_r \circ f = k_r$, $\Int(M,\mathcal{C})$ is $\Aut^{loc}$-invariant.

\subsection{Gromov's results for absolute parallelisms}
\label{ss:parallelism}
In this section, we prove Theorem \ref{thm:frobenius} and Theorem \ref{thm:isloc_orbit} in the special case where the Cartan geometry is simply the data of an absolute parallelism on the manifold $M$. The proof of Theorem \ref{thm:frobenius} consists in exhibiting a local Killing field and we do it by building its graph using standard techniques of integration of an involutive distribution.

The geometric structure defined by a parallelism is an elementary case of Cartan geometry. Let us explain it briefly.

Consider $M$ an $n$-dimensional manifold endowed with an abolute parallelism $\mathcal{P}$, \textit{ie} a set of $n$ vector fields $(X_1,\ldots,X_n)$ on $M$ everywhere linearly independent. Then, let $\omega = (\omega_1,\ldots,\omega_n)$ be the dual parallelism of the cotangent bundle. The data of $\mathcal{P}$ is equivalent to the one of the $1$-form $\omega$, that takes values in $\R^n$. Let $\hat{M} = M$ and $\pi : \hat{M} \rightarrow M$ be the trivial principal fibration with group $\{0\}$, $G$ be the additive group $\R^n$ and $P = \{0\}$. Then, the $1$-form $\omega$ obviously satisfies the axioms of a Cartan connection on this bundle. We will say that the Cartan geometry $(M,\hat{M},\omega)$ built above is \textit{the} Cartan geometry associated to the parallelism $\mathcal{P}$. Moreover, local automorphisms of this Cartan geometry are exactly the local automorphisms of $(M,\mathcal{P})$, \textit{ie} local diffeomorphisms of $M$ commuting with the fields $X_1, \ldots, X_n$.

Conversely, if $(M,\hat{M},\omega)$ is a Cartan geometry with model space $G/P$ with $G = \R^n$ and $P=\{0\}$, then we identify $\hat{M}$ with $M$, $\omega$ will be viewed as a parallelism of $T^*M$ and we associate to it the dual parallelism of $TM$.

Since the total space of this fibration coincides with the base, we omit in this section the notation with ``hats''.

From now on, fix $(M,\mathcal{P})$ an $n$-dimensional manifold together with a parallelism defined by the vector fields $(X_1,\ldots,X_n)$, and note $(M,\hat{M},\omega)$ its associated Cartan geometry. All of the definitions and properties we have established above make sense in this geometric context (curvature, domain of integrability, Killing generators ..). Recall that if $X \in \R^n$, we note $\tilde{X} = \omega^{-1}(X)$ the $\omega$-constant vector field associated to $X$. Those fields are linear combinations of the $X_i$'s. At last, there exists maps $\gamma_{ij}^k$ such that : 
\begin{equation*}
\forall 1\leq i,j\leq n,\: [X_i,X_j] = \sum_k \gamma_{ij}^k X_k.
\end{equation*}
Therefore, the curvature map $K$ is defined for all $i,j$ by
\begin{equation*}
K(x)(e_i,e_j) =[e_i,e_j] - \omega_x([X_i,X_j]) = -(\gamma_{ij}^1(x),\ldots,\gamma_{ij}^n(x)).
\end{equation*}

\subsubsection{Integrating Killing generators}
\label{sss:frobenius_parallelism}

The problem is the following : we are given $A \in \Kill^{n+1}(x_0)$ a Killing generator of order $n+1$ at a point $x_0 \in \Int(M,\mathcal{P})$, and we want to find a vector field defined on a neighbourhood of $x_0$ that commutes with the $n$ vector fields of $\mathcal{P}$. Since this problem is local, without loss of generality, we assume that $M=U \subset \R^n$ is a connected open subset such that $\Int(M,\mathcal{P})=M$, or equivalently such that the maps $k_1, \ldots , k_{n+2}$ are constant over $M$. The parallelism $\mathcal{P}$ provides us a particular trivialisation of $TM$ which is the map $\phi : (x,v)\in TM \mapsto (x,\omega_x(v)) \in M \times \R^n$. We have chosen to make all the computations through this trivialisation. Thus, a vector field will be seen as a smooth map $M \rightarrow \R^n$. If $f$ is a local diffeomorphism, we note $f_* : M \times \R^n \rightarrow M \times \R^n$ the action of its differential map, read in the trivialisation $\phi$.

To exhibit the local Killing field of Theorem \ref{thm:frobenius}, we will build its graph in $M \times \R^n$. Identify canonically $T_{(x,u)}(M \times \R^n)$ with $T_x M \times \R^n$, note $(\partial_1,\ldots,\partial_n)$ the standard basis of the second factor $\R^n$ and consider $\Delta$ the $n$-dimensional distribution on $M \times \R^n$ defined by 
\begin{equation*}
\Delta_{(x,u)} = \Vect(X_i - \sum_{jk} u_j\gamma_{ij}^k \, \partial_k \, , \: 1\leq i \leq n).
\end{equation*}
The introduction of this object is motivated by the following lemma.

\begin{lem}
\label{lem:killing_field}
Let $X$ be a vector field on $M$. Then $X$ is a Killing field if and only if its graph in $M \times \R^n$ integrates the distribution $\Delta$.
\end{lem}

\begin{proof}
Let $\Gamma_X = \{(x,u_1(x),\ldots,u_n(x))=(x,X(x)), \ x \in M \}$ be the graph of $X$, read in the trivialisation $\phi$. Then, for all $x\in M$, the $n$-dimensional tangent space $T_{(x,X(x))}\Gamma_X$ is spanned by the $X_i(x) + \sum_k (X_i.u_k)(x)\partial_k \in T_x M \times \R^n$, $1 \leq i \leq n$. 

When we develop the equations $[X_i,\sum_j u_jX_j](x) = 0$ for $1 \leq i \leq n$, we get that $X$ is a Killing field if and only if for all $1 \leq i,k \leq n$
\begin{equation*}
(X_i.u_k)(x) + \sum_j u_j(x) \gamma_{ij}^k(x)=0.
\end{equation*}
Thus, if $X$ is a Killing field, the tangent space $T_{(x,X(x))}\Gamma_X$ is in fact spanned by the
\begin{equation*}
X_i(x) - \sum_{jk} u_j(x) \gamma_{ij}^k(x) \partial_k, \ 1 \leq i \leq n,
\end{equation*}
meaning $T_{(x,X(x))}\Gamma_X = \Delta_{(x,X(x))}$. Conversely, if we assume $T_{(x,X(x))}\Gamma_X = \Delta_{(x,X(x))}$, then for all $i$, $T_{(x,X(x))}\Gamma_X$ contains $X_i(x) - \sum_{jk} u_j(x) \gamma_{ij}^k(x) \partial_k$. Therefore, we must have for all $i,k$, $(X_i.u_k)(x) = \sum_j u_j(x) \gamma_{ij}^k(x)$.
\end{proof}

\begin{lem}
\label{lem:melnick}
Let $x \in \Int(M,\mathcal{P})$. Let $X \in \R^n$ in the injectivity domain of $\exp_x$, $A \in \Kill^s(x)$ and $\gamma(t) = \exp_x(tX)=\phi_{\tilde{X}}^t(x)$, $t \in [-1,1]$. Assume that for some $s \geq 1$, the maps $k_s$ and $k_{s+1}$ are constant and equal in a neighbourhood of  $\gamma$. Then, we have for all $t$
\begin{equation*}
\left ( \phi_{\tilde{X}}^t \right )_* \! A \in \Kill^s (\gamma(t)).
\end{equation*}
\end{lem}

\begin{proof}
This result corresponds to Proposition 6.2 of \cite{melnick}, with a slight modification. For the sake of completeness, we give a detailed proof in Appendix \ref{appendix}.
\end{proof}

We can now begin the proof. Let $x_0 \in M$ and $A \in \Kill^{n+1}(x_0)$. There exists an integer $s_0 \leq n+1$ such that
\begin{equation*}
k_1(x_0) > \cdots >k_{s_0}(x_0)=k_{s_0 +1}(x_0).
\end{equation*}
Since the maps $k_1, \ldots ,k_{n+2}$ are constant on $M$, the maps $k_{s_0}$  and $k_{s_0+1}$ are constant and equal. We define
\begin{equation*}
\Sigma = \{(x,v) \: | \: v \in \Kill^{s_0}(x) \} \subset M \times \R^n.
\end{equation*}
Note that for an arbitrary local Killing field $X$ on $M$, we must have $X(x) \in \Kill^{\infty}(x)$ everywhere $X$ is defined. Therefore the graph we are looking for has to be included in $\Sigma$ and must contain the point $(x_0,A)$.

\begin{lem}
\label{lem:sigma_manifold}
$\Sigma$ is a submanifold of $M \times \R^n$
\end{lem}

\begin{proof}
Take $x \in M$ and consider $B=B(0,\delta) \subset \R^n$ a ball (for an arbitrary norm) such that $\exp_x$ is defined and injective on $B$. If $Y \in B$ and $y = \exp_x(Y)$, since $\Kill^{s_0}(y)$ and  $\Kill^{s_0}(x)$ have same dimension, Lemma \ref{lem:melnick} implies $\Kill^{s_0}(y) = (\phi_{\tilde{Y}}^1)_* \Kill^{s_0}(x)$. Thus, $\Kill^{s_0}(x)$ depends smoothly of $x$ and $\Sigma$ is indeed a submanifold of $M \times \R^n$.
\end{proof}

We now prove that $\Delta$ can be restricted into a distribution on $\Sigma$, and that this restricted distribution is involutive. From now on for the reader's convenience, we note $V_i =V_i(x,u_1,\ldots,u_n)= \sum_{jk} u_j\gamma_{ij}^k \, \partial_k$.

\begin{lem}
\label{lem:delta_tangent}
The distribution $\Delta$ is everywhere tangent to $\Sigma$.
\end{lem}

\begin{proof}
According to Lemma \ref{lem:melnick}, $(\phi_{X_i}^t)_* : M \times \R^n \rightarrow M \times \R^n$ preserves $\Sigma$ if $t$ is small enough. Take $(x,u_1,\ldots,u_n)=(x,u) \in \Sigma$ and let $c(t) = (\phi_{X_i}^t)_* (x,u)$. Then $\ddt{t=0} c(t) \in T_{(x,u)}\Sigma$. Let us compute this tangent vector. According to the definition of $\phi$, we have
\begin{equation*}
\ddt{t=0} c(t) = (X_i(x),\sum_j u_j \ddt{t=0} \omega((T_x \phi_{X_i}^t) X_j(x))).
\end{equation*}
Moreover,		
\begin{equation*}
\ddt{t=0} \omega((T_x \phi_{X_i}^t) X_j(x)) = \left ( \L_{X_i}\omega \right )_x(X_j(x)) = (\d \omega)_x(X_i,X_j),
\end{equation*}
since $\omega(X_i)$ is constant equal to $e_i$. The curvature form of the Cartan geometry associated to the parallelism $\mathcal{P}$ is $\Omega = \d \omega + \frac{1}{2} [\omega,\omega] = \d \omega$ since the Lie algebra $\g=\R^n$ is abelian. Therefore, $(\d \omega)_x(X_i,X_j) = \Omega_x(X_i,X_j) = K(x)(e_i,e_j) = - (\gamma_{ij}^1(x),\ldots,\gamma_{ij}^n(x))$. Finally, we get
\begin{equation*}
\ddt{t=0} c(t) = X_i(x) - \sum_{jk} u_j \gamma_{ij}^k(x) \partial_k = X_i - V_i \in T_{(x,u)} \Sigma.
\end{equation*}
Thus, for all $i$, $X_i - V_i$ is tangent to $\Sigma$ and we have $\Delta_{(x,u)} \subset T_{(x,u)} \Sigma$.
\end{proof}

From now on, we will consider $\Delta$ as a distribution in $T\Sigma$. 

\begin{lem}
\label{lem:delta_involutive}
When restricted to $\Sigma$, the distribution $\Delta$ is involutive.
\end{lem}

\begin{proof}
Let us compute the bracket of $X_1 - V_1$ and $X_2-V_2$. We get : 

\begin{align*}
[X_1,V_2] & = \sum_{jk} (X_1.\gamma_{2j}^k)u_j \, \partial_k \\
[V_1,X_2] & = - \sum_{jk} (X_2.\gamma_{1j}^k)u_j \, \partial_k \\
[V_1,V_2] & = \sum_{jklm} \gamma_{1j}^k \gamma_{2m}^l (\delta_{km}u_j \, \partial_l - \delta_{jl}u_m \, \partial_k), \text{ where } \delta_{ij} = 1 \text{ if } i=j \text{, } 0 \text{ else}.
\end{align*}
Note $C_p$ the component of $[X_1-V_1,X_2-V_2]$ on $\partial_p$, for $1\leq p\leq n$. 
\begin{equation*}
C_p = \sum_j\left (\sum_k \gamma_{j1}^k\gamma_{k2}^p + \gamma_{2j}^k\gamma_{k1}^p \right )u_j - \sum_j (X_1.\gamma_{2j}^p) u_j + \sum_j (X_2.\gamma_{1j}^p) u_j.
\end{equation*}
The Jacobi relation involving $X_1$, $X_2$ and $X_j$ gives us :
\begin{equation*}
\forall p,\, \forall j, \: \sum_k \gamma_{j1}^k\gamma_{k2}^p + \gamma_{2j}^k\gamma_{k1}^p + \gamma_{12}^k\gamma_{kj}^p = X_j.\gamma_{12}^p + X_1.\gamma_{2j}^p + X_2.\gamma_{j1}^p.
\end{equation*}
Then we have :
\begin{align*}
C_p & = \sum_j \left [(X_j.\gamma_{12}^p) u_j-\sum_k \gamma_{12}^k \gamma_{kj}^p u_j \right ] \\
    & = \left (\sum_j u_j X_j \right ).\gamma_{12}^p - \sum_k \gamma_{12}^k \sum_j \gamma_{kj}^p u_j
\end{align*}
Therefore,
\begin{align*}
[X_1-V_1,X_2-V_2] & = [X_1,X_2] + \sum_p C_p \partial_p \\
                  & = \sum_k \gamma_{12}^k X_k -\sum_p \left (\sum_k \gamma_{12}^k \sum_j \gamma_{kj}^p u_j\right )\, \partial_p + \sum_p \left (\sum_j u_j X_j \right )\!.\gamma_{12}^p \, \partial_p \\
                  & = \sum_k \gamma_{12}^k \left (X_k - \sum_{jp}\gamma_{kj}^p u_j\, \partial_p \right ) + \sum_p \left (\sum_j u_j X_j \right )\!.\gamma_{12}^p \, \partial_p \\
                  & = \underbrace{\sum_k \gamma_{12}^k (X_k - V_k)}_{\in \Delta} + \sum_p \left (\sum_j u_j X_j \right )\!.\gamma_{12}^p \, \partial_p
\end{align*}
Since $(x,u) \in \Sigma$, $(u_1,\ldots,u_n) \in \Kill^{s_0}(x)$ is a Killing generator of order $s_0 \geq 1$ at $x$, we get $\left [(u_1 X_1 + \cdots + u_nX_n). K\right ] (x) = 0$. Therefore, the second term of the sum is $0$. Hence $[X_1-V_1, X_2-V_2] \in \Delta$. Of course, this calculation is still valid for an arbitrary bracket $[X_i-V_i,X_j-V_j]$ and we have proved that $\Delta$ is involutive.
\end{proof}

\paragraph*{Conclusion}
The distribution $\Delta$ gives rise to a foliation $\mathcal{F}$ in $\Sigma$. Consider the leaf $\mathcal{F}_0$ containing $(x_0,A)$ and note $p : \Sigma \rightarrow M$ the projection on the first factor. Since $X_1,\ldots,X_n$ are everywhere in the image of the tangent map of $p|_{\mathcal{F}_0}$, this map is a local diffeomorphism from $\mathcal{F}_0$ to $M$ by the inverse mapping theorem. In a small enough neighbourhood of $x_0$, set $\hat{A}(x)=p|_{\mathcal{F}_0}^{-1}(x)$. Then $\hat{A}$ is a field near $x_0$ whose graph integrates the distribution $\Delta$. Therefore $\hat{A}$ is a Killing field and $\hat{A}(x_0) = A$.

\subsubsection{Orbits of local automorphisms of a parallelism}
\label{sss:parallelism_islocorbit}

We still consider an $n$-dimensional manifold $M$ endowed with an absolute parallelism $\mathcal{P}$, and note $\omega$ the Cartan connection of its associated Cartan geometry. We shall now prove Theorem \ref{thm:isloc_orbit} in this situation. Set $\Omega=\Int(M,\mathcal{P})$.

We first study the $\Kill^{loc}$-orbits in $\Omega$. Fix $x_0 \in \Omega$ and note $\mathcal{O}$ its $\Kill^{loc}$-orbit. The orbit $\mathcal{O}$ must be included in the connected component $U$ of $\Omega$ containing $x_0$. The map $k_{n+1}(x)$ is constant equal to $k_0$ over $U$. This means that $\Kill^{n+1}(x)$ has constant dimension $k_{n+1}^0$ on $U$. Thus, the map $\mathcal{D}^n K : U \rightarrow \mathcal{W}^n$ has constant rank $n-k_0$ since $\Kill^{n+1}(x) = \omega_x(\Ker T_x \mathcal{D}^nK)$ for all $x \in U$. Therefore, its level sets are properly embeded submanifolds of $U$. Let $\mathcal{F} \subset U$ be the connected component of the level set $\mathcal{D}^n K^{-1}(\mathcal{D}^n K(x_0))$ in $U$ containing $x_0$. Note first that $\mathcal{F}$ is preserved by the flows of local Killing fields of $\mathcal{P}$. By construction, Theorem \ref{thm:frobenius} ensures that for all $x \in \mathcal{F}$, $T_x \mathcal{F} = \Kill^{loc}(x)|_x = \{X(x), \ X \in \Kill^{loc}(x)\}$ since we have $\omega_x(T_x \mathcal{F}) = \omega_x(\Ker T_x \mathcal{D}^n K)=\Kill^{n+1}(x)$ and $x \in \Int(M,\mathcal{P})$. Therefore, for all $x \in \mathcal{F}$, the $\Kill^{loc}$-orbit of $x$ is contained in $\mathcal{F}$ and has dimension $k_0 = \dim \mathcal{F}$. By connectedness, $\mathcal{O} = \mathcal{F}$.

We conclude that for all $x \in U$, the $\Kill^{loc}$-orbit of $x$ is exactly the connected component containing $x$ of the level set of $\mathcal{D}^n K$ passing through $x$.

Now, consider $\mathcal{O}'$ an $\Aut^{loc}$ orbit in $\Omega$. It is enough to prove that the intersection between $\mathcal{O}'$ and any connected component of $\Omega$ is a closed submanifold of this component. Let $U$ be a connected component and  $x \in U\cap \mathcal{O}'$. Then $\mathcal{O}' \cap U$ is included in $\mathcal{D}^n K^{-1}(\mathcal{D}^n K(x)) \cap U$. Since $\mathcal{O}'$ is stable under the action of flows of local Killing fields, what has been proved above implies that $\mathcal{O}' \cap U$ is a union of connected component of $\mathcal{D}^n K^{-1}(\mathcal{D}^n K(x)) \cap U$, and therefore it is a closed submanifold of $U$.

\subsection{From parallelisms to general Cartan geometries}
\label{ss:generalcase}
Let $(M,\mathcal{C})$ be a manifold endowed with a Cartan geometry modeled on $\X = G/P$. We still note $\pi : \hat{M} \rightarrow M$ the associated $P$-principal bundle and $\omega$ the Cartan connection. The Cartan form defines an absolute parallelism on $\hat{M}$, denoted $\mathcal{P}_{\omega}$ : identify once for all in an arbitrary way the vector space $\g$ with $\R^m$ and note $(e_1,\ldots,e_m)$ its standard basis. Let $X_i = \omega^{-1}(e_i)$, for $1 \leq i \leq m$ and set $\mathcal{P}_{\omega} = (X_1, \ldots , X_m)$. 

Thus, we have two different geometric structures on two different manifolds : a Cartan geometry on $M$, namely $(M,\mathcal{C})$, and a parallelism on $\hat{M}$, namely $(\hat{M},\mathcal{P}_{\omega})$. We then have two curvature maps :
\begin{enumerate}
\item The curvature map $K : \hat{M} \rightarrow \Hom(\Lambda^2 \g,\g) \simeq \Hom(\Lambda^2 \R^m,\R^m)$ of the Cartan geometry $(M,\mathcal{C})$ ;
\item The curvature map $K^{\mathcal{P}_{\omega}} : \hat{M} \rightarrow \Hom(\Lambda^2 \R^m,\R^m)$ of $(\hat{M},\mathcal{P}_{\omega})$, as defined in section \ref{ss:parallelism}.
\end{enumerate}
Those maps satisfy the relation $K(\hx) = [\:,]_{\g} + K^{\mathcal{P}_{\omega}}(\hx), \text{ in } \Hom(\Lambda^2 \R^m,\R^m)$, for all $\hx \in \hat{M}$, since we have for all $X,Y \in \g$, $K(\hx)(X,Y) = [X,Y]_{\g} - \omega_{\hx}([\tilde{X},\tilde{Y}])$. Therefore, for all $r \geq 1$, $D^r K = D^r K^{\mathcal{P_{\omega}}}$ as maps on $\hat{M}$. 

We deduce that for all $\hat{x} \in \hat{M}$ and $A \in \g$, $A$ is a Killing generator of order $r \geq 1$ at $\hat{x}$ of $(M,\mathcal{C})$ if and only if it is a Killing generator of order $r$ at $\hx$ of $(\hat{M},\mathcal{P_{\omega}})$. Consequently, the integrability domain of $(\hat{M},\mathcal{P}_{\omega})$ coincides with $\pi^{-1}(\Int(M,\mathcal{C}))$.

\subsubsection{Integrability theorem}
\label{sss:generalcase_frobenius}

We first explain how to deduce from the previous study the proof of Theorem \ref{thm:frobenius} on the integrability of Killing generators in a general Cartan geometry. It is the easiest part.

Take $\Omega=\Int(M,\mathcal{C})$. Fix $x \in \Omega$, $\hx \in \hat{M}$ lying over $x$ and $A \in \Kill^{m+1}(\hx)$ a Killing generator of order $m+1$ of $(M,\mathcal{C})$ at $\hx$. According to what has been said above, $\hx \in \Int(\hat{M},\mathcal{P}_{\omega})$ and $A$ is a Killing generator of order $m+1$ of $(\hat{M},\mathcal{P}_{\omega})$ at $\hx$. Then there exists $\hat{A}$ a local Killing field at $\hx$ of $(\hat{M},\mathcal{P}_{\omega})$, defined on an open set $\mathcal{U} \subset \hat{M}$ and such that $\omega_{\hx}(\hat{A})=A$. Shrinking $\mathcal{U}$ if necessary, we can apply Proposition \ref{prop:extension} to conclude that $\hat{A}$ extend to the saturated open set $\mathcal{U}.P$ and that this new vector field is the lift of a local Killing field of $(M,\mathcal{C})$ defined near $x$.

\subsubsection{Structure of orbits} 
\label{sss:generalcase_gromov}

The Theorem \ref{thm:isloc_orbit} on $\Aut^{loc}$-orbits will be more difficult to prove. The study of the orbits of a parallelism ensures that in $\pi^{-1}(M^{int})$, the orbits of the parallelism $\mathcal{P}_{\omega}$ are closed submanifolds. Moreover, Proposition \ref{prop:extension} implies that orbits of $(M,\mathcal{C})$ are the projection on $M$ of orbits of $(\hat{M},\mathcal{P}_{\omega})$. We are then left to prove that these projections are closed submanifolds. This will not be true if we consider orbits in the whole domain $\Int(M,\mathcal{C})$ and we will fix this problem by shrinking this open dense set.

Recall that we need to assume that the model space $\X = G/P$ is of algebraic type. In particular, the group $\Ad_{\g}(P)<\GL(\g)$ acts algebraically on $\g$ and the $P$-orbits in every space $\mathcal{W}^{r}$ are properly embeded submanifolds. We will make use of the following property of algebraic actions.

\begin{prop}
\label{prop:algebraic_orbit}
Let $G$ be a real algebraic group and $G\times V \rightarrow V$ an algebraic action of $G$ on a real variety $V$. Let $v \in V$ and $v'$ in the boundary of the orbit $G.v$. Then, $\dim G.v' < \dim G.v$.
\end{prop}

Let $\Phi : \hx \in \pi^{-1}(M^{int}) \mapsto \mathcal{D}^m K(\hx) \in \mathcal{W}^m$. Define $d : x\in M^{int} \mapsto \dim(P.\Phi(\hx))$, for any $\hx$ lying over $x$ (recall $\Phi$ is $P$-equivariant). The map $d$ is lower semi-continuous and $\Aut^{loc}$-invariant since $w \in \mathcal{W}^m \mapsto \dim P.w$ is itself lower semi-continuous and $\Phi \circ \hat{f} = \Phi$ for all $f \in \Aut^{loc}$. Therefore, there exists $\Omega \subset M^{int}$ an $\Aut^{loc}$-invariant open dense subset on which $d$ is locally constant. Since $k_{m+1}$ and $d$ are locally constant on this open set, $\Omega$ decomposes into a finite union of open sets $\Omega = \Omega_1 \cup \cdots \cup \Omega_l$ such that for all $1 \leq i \leq l$, the maps $k_{m+1}$ and $d$ are constant equal to $k_{m+1}^i$ and $d_i$ on $\Omega_i$. If we assume this decomposition to be minimal, each $\Omega_i$ will be $\Aut^{loc}$-invariant. We now focus on the behaviour of local orbits in each $\Omega_i$ and prove that for all $i$, the $\Aut^{loc}$-orbits in $\Omega_i$ are closed submanifold with the same dimension.

Let $\hat{\Omega}_i = \pi^{-1}(\Omega_i)$. We note $\Phi_i = \Phi|_{\hat{\Omega}_i}$. Then $\Phi_i$ has constant rank $m-k_{m+1}^i$ on $\hat{\Omega}_i$. Let $\mathcal{O}$ be the $\Aut^{loc}$-orbit of a point $x \in \Omega_i$. Let $\mathcal{P}_i$ be the restriction to $\hat{\Omega}_i$ of the parallelism $\mathcal{P}_{\omega}$ defined by the Cartan connection. Take $\hx$ lying over $x$ and note $\mathcal{O}'$ the $\Aut^{loc}$-orbit of $\hx$ for the structure $(\hat{\Omega}_i,\mathcal{P}_i)$. As we saw in section \ref{sss:parallelism_islocorbit}, $\mathcal{O}'$ is in fact a union of connected components of the level set $\Phi_i^{-1}(\Phi_i(\hx))$. According to Proposition \ref{prop:extension}, $\mathcal{O} = \pi(\mathcal{O}')$. This observation invites us to study the union of the $\Aut^{loc}$-orbits of every point $\hx$ in the fiber $\pi^{-1}(x)$, that is $\pi^{-1}(\mathcal{O})$.

Fix $\hx$ in $\pi^{-1}(x)$, note $w = \Phi_i(\hx)$. We claim that $\Phi_i^{-1}(P.w)=\pi^{-1}(\mathcal{O})$ is a closed submanifold of $\hat{\Omega}_i$, saturated by the $P$-action.

First, $\Phi_i^{-1}(P.w)$ has the structure of a submanifold since $\Phi_i$ has constant rank and $P.w$ is a submanifold of $\mathcal{W}^m$ included in the image of $\Phi_i$. Note that $\dim \Phi_i^{-1}(P.w) = \dim P.w + (\dim \hat{\Omega}_i - \text{Rk} \, \Phi_i ) = d_i + k_{m+1}^i$.

Let $\hx_n$ be a sequence of $\Phi_i^{-1}(P.w)$ that converges to a point $\hx_{\infty}$ in $\hat{\Omega}_i$. The limit $\Phi_i(\hx_{\infty})$ belongs to the closure of $P.w$ for the Hausdorff topology in $\mathcal{W}^m$, and \textit{a fortiori} to its Zariski closure. Since the $P$-action on $\mathcal{W}^m$ is algebraic, if $\Phi_i(\hx_{\infty})$ were not in $P.w$, Proposition \ref{prop:algebraic_orbit} would imply $\dim(P.\Phi_i(\hx_{\infty})) < d_i$, contradicting $\hx_{\infty} \in \hat{\Omega}_i$. Then, $\hx_{\infty} \in \Phi_i^{-1}(P.w)$.

Finally, since the projection $\pi$ is open, $\pi(\Phi_i^{-1}(P.w))$ is a closed submanifold of $\Omega_i$. Moreover, $\pi(\Phi_i^{-1}(w)) = \pi(\Phi_i^{-1}(P.w))$. Since the orbit $\mathcal{O}'$ is a union of connected components of $\Phi_i^{-1}(w)$ and since $\mathcal{O} = \pi(\mathcal{O}')$, $\mathcal{O}$ is a union of connected components of the manifold $\pi(\Phi_i^{-1}(P.w))$. Thus $\mathcal{O}$ is itself a closed submanifold of $\Omega_i$. Moreover, $\dim \mathcal{O} = \dim \Phi_i^{-1}(P.w) - \dim P = d_i + k_{m+1}^i - \dim \p$. This dimension does not depend on the orbit $\mathcal{O} \subset \Omega_i$.

Remark that $\dim P.w + \dim \Kill^{m+1}(\hx) \geq \dim P$, since $\Lie(\Stab_P(w)) = \Kill^{m+1}(\hx) \cap \, \p$.

\subsection{Extension to more general structures}
\label{ss:extension}
If we are interested in the orbits of automorphisms of a given Cartan geometry that also preserve some tensor, we cannot \textit{a priori} find a Cartan geometry that models the structure resulting from the union of the initial Cartan geometry and the tensor. We have in mind the proof of a theorem of D'Ambra whose statement is :

\begin{thm*}[\cite{dambra_gromov}, 0.12.A]
Let $(M,g)$ be a simply connected, real-analytic, compact, Lorentz maninold. Then the isometry group $\Isom(M,g)$ is compact.
\end{thm*} 

D'Ambra used Gromov's result on $(M,\varphi)$, where $\varphi$ is the structure resulting from the \textit{union} of $g$ and a well chosen set of vector fields $(X_1,\ldots,X_s)$. That is, in an open dense subset of $M$, the orbits of local automorphisms that preserve the metric $g$ \textit{and} each field $X_i$ are closed submanifolds.

Also, one can remark that Theorem \ref{thm:isloc_orbit} has an awkward weakness with regard to Gromov's statement, since $A$-rigid structures include structures similar to $\varphi$, whereas Cartan geometries do not. Thus, the motivation of this section is to fill this gap and to obtain a formulation of Theorem \ref{thm:isloc_orbit} for geometric structures more general than Cartan geometries. 

\subsubsection{Cartan geometries with additional geometric structure}
\label{sss:generalized_cartan}

In any Cartan geometry, the tangent bundle $TM$ is naturally isomorphic to $\hat{M}\times_{P} \g / \p$, where $P$ acts on $\g / \p$ via the representation $\bar{\Ad}$ (see \cite{sharpe}, p.188). Therefore, many geometrical objects on $M$ can be interpreted as $P$-equivariant maps from $\hat{M}$ to some manifold endowed with a $P$-action. For instance, one can see tensors of type $(r,s)$ on $M$ as $P$-equivariant maps $\hat{M} \rightarrow \mathcal{T}_r^s(\g / \p)$, where $\mathcal{T}_r^s(\g / \p)$ denotes the space of tensors of type $(r,s)$ on $\g / \p$, endowed with the action induced by $\bar{\Ad}_{\g / \p}(P)$. Moreover, any $k$-dimensional distribution on $M$ can be seen has a $P$-equivariant map $\hat{M} \rightarrow Gr_k(\g / \p)$, where $Gr_k(\g / \p)$ denotes the $k$-dimensional Grassmanian of $\g / \p$, $P$ acting on it via the adjoint representation.

\begin{dfn}
Let $\X = G/P$ be a model space and $W$ be a manifold endowed with a right action of the group $P$. A generalized Cartan geometry $(M,\mathcal{C},\varphi)$ of type $(\X , W)$ on a manifold $M$ is the data of a Cartan geometry $(M,\mathcal{C})$ modeled on $\X$ together with a $P$-equivariant map $\varphi : \hat{M} \rightarrow W$.

Local automorphisms of such a geometric structure are local automorphisms $f : U \rightarrow V$ of $(M,\mathcal{C})$ whose lift $\hat{f}$ satisfies $\varphi \circ \hat{f} = \varphi$ on $\pi^{-1}(U)$. A local Killing field of $(M,\mathcal{C},\varphi)$ at a point $x \in M$ is a vector field $X$ defined on a neighbourhood of $x$ whose flow is composed with local automorphisms of $(M,\mathcal{C},\varphi)$.
\end{dfn}

\begin{rem}
In the examples given above, a $\varphi$-preserving automorphism is nothing more than an automorphism of $(M,\mathcal{C})$ preserving the corresponding tensor or $k$-dimensional distribution on $M$.
\end{rem}

In what follows, $\Aut^{loc}_{\varphi}$ will denote the local automorphisms of a generalized Cartan geometry $(M,\mathcal{C},\varphi)$, and $\Kill^{loc}_{\varphi}(x)$ its local Killing fields defined near $x \in M$. If $X \in \Kill^{loc}_{\varphi}(x)$ is defined on $U \subset M$, $X$ admits a lift $\hat{X}$ to $\pi^{-1}(U)$ such that 
\begin{enumerate}
\item $\hat{X}$ commutes with the right $P$-action ;
\item $\L_{\hat{X}}\omega = 0$ ;
\item and, if $\phi^t_{\hat{X}}$ denotes its local flow, $\varphi \circ \phi_{\hat{X}}^t = \varphi$ on $\pi^{-1}(U)$.
\end{enumerate}
Conversely, a vector field on $\pi^{-1}(U)$ satisfying these three conditions is the lift of a local Killing field of $(M,\mathcal{C},\varphi)$.

\subsubsection{Generalization of the results}

Theorem \ref{thm:isloc_orbit} naturally extends to this kind of geometric structures, when we assume $P$ and $W$ algebraic and the $P$-action on $W$ algebraic.

\begin{thm}
\label{thm:gromov_generalized}
Let $(M,\mathcal{C},\varphi)$ be a generalized Cartan geometry of type $(\X,W)$. Assume that $\X$ is of algebraic type, the group $P$ is algebraic, $W$ is algebraic and the $P$-action on $W$ is algebraic. Then, there exists an open dense subset $\Omega$ of $M$ on which the $\Aut^{loc}_{\varphi}$-orbits are closed submanifolds. 

More precisely, $\Omega$ admits a decomposition $\Omega = \Omega_1 \cup \cdots \cup \Omega_k$ where each $\Omega_i$ is an open subset, preserved by local automorphisms, and in which the $\Aut^{loc}_{\varphi}$-orbits are closed submanifolds which all have the same dimension.
\end{thm}

The proof of this result will also be based on an integrability result. Before starting it, we need some additional definitions.

\begin{dfn}
Let $W$ be a manifold endowed with a right $P$-action. We associate to it a vector bundle $B(W)$ over $W$ defined by
\begin{equation*}
B(W) = \{ (w,\alpha) , \ w \in W, \ \alpha \in \Hom(\g,T_w W)\}.
\end{equation*}
If $\varphi : \hat{M} \rightarrow W$ is a $P$-equivariant map, we define a map $\mathcal{T}\varphi : \hat{M} \rightarrow B(W)$ setting for all $\hx \in \hat{M}$
\begin{equation*}
\mathcal{T}\varphi(\hx) = (\varphi(\hx) , T_{\hx} \varphi \circ \omega_{\hx}^{-1}) \in B(W).
\end{equation*}
The map $\mathcal{T}\varphi$ is $P$-equivariant if we endowed $B(W)$ with the $P$-action given by
\begin{equation*}
(w,\alpha).p = (w.p, T_w R_p \circ \alpha \circ \Ad(p^{-1})),
\end{equation*}
where $R_p$ denotes the right action of $P$ on $W$. We then define recursively for $r \geq 0$
\begin{align*}
B^0(W) & = W \text{ and } \mathcal{T}^{\: 0} \varphi = \varphi \\
B^{r+1}(W) & = B(B^r(W)) \text{ and } \mathcal{T}^{\: r+1} \varphi = \mathcal{T}(\mathcal{T}^{\: r} \varphi).
\end{align*}
The group $P$ acts on the right on each $B^r(W)$ and the map $\mathcal{T}^{\: r} \varphi : \hat{M} \rightarrow B^r(W)$ is $P$-equivariant.
\end{dfn}

\begin{rem}
In the special case where $W$ is a vector space, we have, after natural identifications, $\mathcal{T}^{\: r} \varphi = \mathcal{D}^r \varphi$ (see Definition \ref{dfn:covariant_derivative}).
\end{rem}

\begin{rem}
Note that $B(W) = (W\times \g)^* \otimes TW$. Since the tangent bundle of a smooth quasi-projective variety is itself smooth and quasi-projective, if we assume that $W$ is algebraic and the $P$-action algebraic, then for all $r \geq 1$, $B^r(W)$ will be algebraic and $P$ will act on it algebraically.
\end{rem}

\begin{rem}
If $f$ is a local automorphism of $(M,\mathcal{C},\varphi)$, we verify that for all $r \geq 1$, $\mathcal{T}^{\: r}\Phi \circ \hat{f} = \mathcal{T}^{\: r} \Phi$, where $\Phi = (K,\varphi)$ and $\hat{f}$ denotes the lift of $f$ ($K$ is the curvature map of $(M,\mathcal{C})$).
\end{rem}

\begin{dfn}[Killing generators]
Let $(M,\mathcal{C},\varphi)$ be a generalized Cartan geometry of type $(\X,W)$. Note $K : \hat{M} \rightarrow V = \Hom(\Lambda^2(\g/\p),\g)$ the curvature map of $(M,\mathcal{C})$ and $\Phi = (K,\varphi) : \hat{M} \rightarrow V \times W$. For $r \geq 1$ and $\hx \in \hat{M}$, we define the space of Killing generators of order $r$ at $\hx$ as
\begin{equation*}
\Kill^r_{\varphi}(\hx) = \omega_{\hx}\left (\Ker T_{\hx}( \mathcal{T}^{\: r-1} \Phi) \right ).
\end{equation*}
These spaces satisfy for all $p \in P$ the relation $\Kill^r_{\varphi}(\hx.p) = \Ad(p^{-1})\Kill^r_{\varphi}(\hx)$ and we define for $x \in M$, $k_r(x) = \dim \Kill^r_{\varphi}(\hx)$, for all $\hx \in \pi^{-1}(x)$.
\end{dfn}

We then have $k_{r+1}(x) \leq k_r(x)$ for all $x \in M$ and $r \geq 1$ and the map $\{x \mapsto k_r(x)\}$ is upper semi-continuous and $\Aut^{loc}_{\varphi}$-invariant for all $r \geq 1$. Let $\Int(M,\mathcal{C},\varphi)$ be the open dense $\Aut^{loc}_{\varphi}$-invariant subset 
\begin{equation*}
\Int(M,\mathcal{C},\varphi) = M^{int}_{\varphi} = \{x \in M \ | \ k_1, \ldots , k_{m+2} \text{ are locally constant at } x \}.
\end{equation*}

We then have an integrability result, similar to Theorem \ref{thm:frobenius}. 

\begin{thm}
\label{thm:frobenius_generalized}
If $x_0 \in M^{int}_{\varphi}$ and $A \in \Kill^{m+1}_{\varphi}(\hx_0)$ for $\hx_0 \in \pi^{-1}(x_0)$, then there exists a local Killing field at $x_0$ whose lift $\hat{A}$ is such that $\omega(\hat{A}(\hx_0))= A$.
\end{thm}

As we will see, the proof reuses essentially the same arguments.

\subsubsection{Sketch of proof}

\paragraph*{Proof of Theorem \ref{thm:frobenius_generalized}}

It is enough to find a field $\hat{A}$ defined on a convex exponential neighbourhood of $\hx_0$ that is a local Killing field of $(\hat{M},\mathcal{P}_{\omega})$ and whose local flow preserves $\varphi$. We then treat the case of an absolute parallelism.

Assume $(M,\mathcal{C})$ is the Cartan geometry associated to a parallelism $(M,\mathcal{P})$. Then, $\hat{M}=M$ and $\varphi$ is a smooth map $M \rightarrow W$. Take $x_0 \in M^{int}_{\varphi}$ and $A \in \Kill^{n+1}_{\varphi} (x_0)$. As in Section \ref{sss:frobenius_parallelism}, we can assume $M=U$ is an open subset of $\R^n$ on which the functions $k_1, \ldots, k_{n+2}$ are constant. We work in the trivialisation of $TM$ given by $\mathcal{P}$ and if $f$ is a local diffeomorphism of $M$, we note $f_* : M \times \R^n \rightarrow M \times \R^n$ its differential action read in the trivialisation given by the Cartan connection $\omega$. There exists $1 \leq s \leq n+1$ such that $k_s = k_{s+1}$ on $M$. 

\begin{lem}
\label{lem:nomizu_generalized}
Let $x_0 \in M$. Assume, for some $s \geq 1$, the maps $k_s$ and $k_{s+1}$ locally constant and equal in a neighbourhood of $x_0$. Then, for $X \in \g$ small enough and $t \in [-1,1]$, we have
\begin{equation*}
\left (\phi_{\tilde{X}}^t\right )_* \Kill^s_{\varphi}(x_0) = \Kill^s_{\varphi}(\phi_{\tilde{X}}^t(x_0)).
\end{equation*}
\end{lem}

\begin{proof}
This result is local. If we fix a chart $\psi$ of $W$ in a neighbourhood of $\varphi(x_0)$ taking values in $\R^{\dim W}$ and if we set $\bar{\varphi} = \psi \circ \varphi$ in a small enough neighbourhood of $x_0$, then we can prove, using local trivialisations of $B^r(W)$, that $\Kill^r_{\varphi} = \Kill^r_{\bar{\varphi}}$ on this neighbourhood. This shows that it is enough to prove the lemma in the special case where $W$ is a vector space.

In this situation, $\Phi =(K,\varphi)$ takes values in a vector space. Up to natural identifications, the maps $\mathcal{T}^{\: r} \Phi$ correspond to the covariant derivatives $\mathcal{D}^r \Phi$. One can then verify that the proof of Lemma \ref{lem:melnick} can directly be adapted when we replace the curvature map $K$ by $\Phi$ and the spaces $\Kill^r(x)$ by $\Kill^r_{\varphi}(x)$.
\end{proof}

Define 
\begin{equation*}
\Sigma_{\varphi} = \{(x,u) \in M \times \R^n \ | \ u \in \Kill^{s}_{\varphi}(x)\}
\end{equation*}
and let $\Delta$ be the distribution on $M \times \R^n$ defined in Lemma \ref{lem:killing_field} :
\begin{equation*}
\Delta_{(x,u)} = \Vect(X_i - \sum_{jk} u_j\gamma_{ij}^k \, \partial_k \, , \: 1\leq i \leq n).
\end{equation*}
Lemma \ref{lem:nomizu_generalized} gives us that $\Sigma_{\varphi}$ is a submanifold of $M \times \R^n$ and that $\Delta$ is everywhere tangent to $\Sigma_{\varphi}$ (the proof is exactly the same than Lemmas \ref{lem:sigma_manifold} and \ref{lem:delta_tangent}). The same calculation than in Lemma \ref{lem:delta_involutive} gives that $\Delta|_{T\Sigma_{\varphi}}$ is involutive since for every $(x,u) \in \Sigma_{\varphi}$, $u \in \Kill^1(x)$ is a Killing generator of order $1$ of $(M,\mathcal{P})$.

The distribution $\Delta$ can be integrated into a foliation of $\Sigma_{\varphi}$. Let $\mathcal{F}_0$ be the leaf in $\Sigma_{\varphi}$ that contains $(x_0,A)$. Note $p : \Sigma \rightarrow M$ the projection on the first factor. The map $p|_{\mathcal{F}_0}$ is a local diffeomorphism from $\mathcal{F}_0$ to $M$ by the inverse mapping theorem. In a small enough neighbourhood of $x_0$, set $\hat{A}(x)=p|_{\mathcal{F}_0}^{-1}(x)$. Then, $\hat{A}$ satisfies :

\begin{enumerate}
\item The graph of $\hat{A}$ integrates the distribution $\Delta$, \textit{ie} $\hat{A}$ is a local Killing field of $(M,\mathcal{P})$, according to Lemma \ref{lem:killing_field}.
\item For every $x$ where $\hat{A}$ is defined, $T_x \varphi (\hat{A}(x)) = 0$, since $(x,\hat{A}(x)) \in \Sigma_{\varphi} \Rightarrow \hat{A}(x) \in \Kill^{1}_{\varphi}(x)$.
\end{enumerate}
The second point implies that $\varphi(\phi^t_{\hat{A}}(x))$ does not depend on $t$ and $\varphi \circ \phi^t_{\hat{A}} = \varphi$ for small $t$ and in a neighbourhood of $x_0$. Then, $\hat{A} \in \Kill^{loc}_{\varphi}(x_0)$ and the theorem is proved.

This last theorem ensures that for all $\hx \in \pi^{-1}(M^{int}_{\varphi})$,
\begin{equation*}
\Kill^{m+1}_{\varphi}(\hx) = \omega_{\hx}\left (\Kill^{loc}_{\varphi}(x)|_{\hx}\right ) := \omega_{\hx}\left (\{\hat{X}(\hx), \ X \in \Kill^{loc}_{\varphi}(x)\}\right ).
\end{equation*}

\paragraph*{Proof of Theorem \ref{thm:gromov_generalized}}
Let $\Phi = (K,\varphi)$ and $\Psi = \mathcal{T}^{\: m} \Phi : \hat{M} \rightarrow B^m(V \times W)$, where $V = \Hom(\Lambda^2(\g / \p),\g)$ denotes the target space of the curvature map $K$. By assumption, $P$ acts algebraically on $V \times W$, and then it also acts algebraically on $B^m(V \times W)$. This map $\Psi$ has locally constant rank over $\pi^{-1}(M^{int}_{\varphi})$ since $\Kill^{m+1}_{\varphi}(\hx) = \omega_{\hx}\left ( \Ker T_{\hx}(\mathcal{T}^{\: m} \Phi) \right )$ and $k_{m+1}$ is locally constant over $M^{int}_{\varphi}$. Define $d : M \rightarrow \N$ by $d(x) = \dim P.\Psi(\hx)$ for all $\hx$ lying over $x$. The map $d$ is upper semi-continuous and $\Aut^{loc}_{\varphi}$-invariant. Consequently, there exists an $\Aut^{loc}_{\varphi}$-invariant open dense subset $\Omega \subset M^{int}_{\varphi}$ on which $d$ and $k_{m+1}$ are locally constant. This set splits into a minimal decomposition $\Omega = \Omega_1 \cup \cdots \cup \Omega_k$ such that every $\Omega_i$ is an open $\Aut^{loc}_{\varphi}$-invariant set on which $k_{m+1}$ and $d$ are constant equal to $k_{m+1}^i$ and $d_i$. Then, for $\hx \in \hat{\Omega}_i$, if $w = \Psi(\hx)$, $\Psi^{-1}(P.w)$ is a closed submanifold of $\hat{\Omega}_i$ (according to Proposition \ref{prop:algebraic_orbit}), saturated by the $P$-action. Then, $\pi(\Psi^{-1}(P.w))$ is a closed submanifold of $\Omega_i$. By construction, we have 
\begin{equation*}
\forall \hx \in \Psi^{-1}(P.w), \ T_{\hx}( \Psi^{-1}(P.w)) = \{\hat{X}(\hx), \ X \in \Kill^{loc}_{\varphi}(x)\}.
\end{equation*}
Then, we can apply the same reasoning as in Section \ref{sss:generalcase_gromov} and conclude that the $\Aut^{loc}$-orbits in $\Omega_i$ are union of connected components of $\pi(\Phi^{-1}(P.w))$, and then are closed submanifolds which all have the same dimension.

\subsection{Another approach of Singer's theorem}

We finish this article by proving that we can view Singer's generalized theorem as a corollary of Theorem \ref{thm:isloc_orbit}, if we assume infinitesimal homogeneity with a higher degree.

Let $(M,\mathcal{C})$ a Cartan geometry, modeled on $\X = G/P$, and $m=\dim G$. Assume that $M$ is connected and $(M,\mathcal{C})$ is $(m+1)$-infinitesimally homogeneous, that is $\mathcal{D}^{m+1} K : \hat{M} \rightarrow \mathcal{W}^{m+1}$ takes values in a single orbit. Then, for all $r \leq m$, $\mathcal{D}^r K$ takes also values in a single $P$-orbit in $\mathcal{W}^r$. Therefore, all the maps $K, \mathcal{D}^{1}K, \ldots ,\mathcal{D}^{m+1}K$ have constant rank on $\hat{M}$, or equivalently $M = M^{int}$. Since $\mathcal{D}^{m+1}K$ is $P$-equivariant, for all $\hx \in \hat{M}$, $\Ker (T_{\hx} \mathcal{D}^{m+1}K) + T_{\hx} (\hx .P) = T_{\hx}\hat{M}$, implying $\pi_* \Ker (T_{\hx} \mathcal{D}^{m+1}K) = T_x M$. Then, Theorem \ref{thm:frobenius} gives us $T_xM = \{X(x), \ X \in \Kill^{loc}(x)\}$ and $(M,\mathcal{C})$ is locally homogeneous.

\section{Appendix : Technical proof}
\label{appendix}
We give here a detailed proof of Lemma \ref{lem:melnick}. Recall its statement

\begin{lem*}[\ref{lem:melnick}]
Let $x \in \Int(M,\mathcal{P})$. Let $X \in \R^n$ in the injectivity domain of $\exp_x$, $A \in \Kill^s(x)$ and $\gamma(t) = \exp_x(tX)=\phi_{\tilde{X}}^t(x)$, $t \in [-1,1]$. Assume that for some $s \geq 1$, the maps $k_s$ and $k_{s+1}$ are constant and equal in a neighbourhood of  $\gamma$. Then, we have for all $t$
\begin{equation*}
\left ( \phi_{\tilde{X}}^t \right )_* \! A \in \Kill^s (\gamma(t)).
\end{equation*}
\end{lem*}

By assumption, $\Kill^{s}(\gamma(t))=\Kill^{s+1}(\gamma(t))$ and these spaces have constant dimension. Set $A_t = (\phi_{\tilde{X}}^t)_* A$. Let $\mathcal{U} \subset \g$ be a neighbourhood of $0$ on which $\exp_x$ realizes a diffeomorphism onto its image. We define a vector field $\chech{A}$ on $\exp(x,\mathcal{U})$ setting $\chech{A}(\exp(x,Y)) = (\phi_{\tilde{Y}}^1)_* A$. Note that for all $f \in \mathcal{C}^{\infty}(M)$ and $t$ we have
\begin{equation*}
(\tilde{A_t}.f)(\gamma(t)) = (\chech{A}.f)(\gamma(t)).
\end{equation*}
Our aim is to prove that for $1 \leq r \leq s$ the map $t \mapsto D^rK(\gamma(t))\lefthalfcup A_t$ identically vanishes. We first compute its derivatives :
\begin{align*}
\left . \frac{\d}{\d t} \right |_{t=t_0} D^rK(\gamma(t))\lefthalfcup A_t 
& = \left . \frac{\d}{\d t} \right |_{t=t_0} (\tilde{A_t}. D^{r-1}K)(\gamma(t))
& = \left . \frac{\d}{\d t} \right |_{t=t_0} (\chech{A}. D^{r-1}K)(\phi_{\tilde{X}}^t(x)) \\
& = (\tilde{X}.\chech{A}.D^{r-1}K)(\gamma(t_0)) & \\
& = (\chech{A}.\tilde{X}.D^{r-1}K)(\gamma(t_0)) & 
\end{align*}
since $[\chech{A},\tilde{X}](\gamma(t)) = 0 \Rightarrow (\L_{\tilde{X}}\L_{\chech{A}}f)(\gamma(t))=(\L_{\chech{A}}\L_{\tilde{X}}f)(\gamma(t))$ for all $f \in \mathcal{C}^{\infty}(M)$.
Then, we have
\begin{align*}
\frac{\d}{\d t} D^rK(\gamma(t))\lefthalfcup A_t & = (\chech{A}.\tilde{X}.D^{r-1}K)(\gamma(t)) \\
                                                & = (\tilde{A_t}.\tilde{X}.D^{r-1}K)(\gamma(t)) \\
                                                & = (D^{r+1}K(\gamma(t)) \lefthalfcup A_t)\lefthalfcup X.
\end{align*}
For $r \geq 1$, let $\mathcal{C}^r_{x}  : \g \rightarrow \mathcal{W}^{r-1}$ be the linear map $\{X \mapsto \mathcal{D}^r K(x) \lefthalfcup X\}$. We finally get
\begin{equation}
\label{diff:equation}
\frac{\d}{\d t} \mathcal{C}^r_{\gamma(t)}(A_t) = \mathcal{C}^{r+1}_{\gamma(t)}(A_t)\lefthalfcup X.
\end{equation}
Let $n_s = \dim \Hom(\otimes^s \R^n, V)$ and $(w_{s,1},\ldots,w_{s,n_s})$ be a basis of $\Hom(\otimes^s \R^n, V)$, for $s \geq 0$. There exists linear forms $f^{i,j}_{x} \in (\R^n)^*$, $i \geq 0$, $1 \leq j \leq n_i$ such that for $r \geq 1$, 
\begin{equation*}
\mathcal{C}	^r_{x} = \sum_{\substack{ 0 \leq i \leq r-1 \\1 \leq j \leq n_i}} f^{i,j}_{x}w_{i,j}.
\end{equation*}
Since $\Kill^r = \Ker \mathcal{C}^r$ for all $r$ and $\Kill^{s}$ has constant dimension along $\gamma$, the rank of the family $(f^{ij}_{\gamma(t)}$, $0\leq i \leq s-1$, $1 \leq j \leq n_i) \subset (\R^n)^*$, does not depend on $t$. At last, $\Kill^{s}(\gamma(t)) = \Kill^{s+1}(\gamma(t))$ implies that for all $0 \leq j \leq n_{s}$,
\begin{equation*}
\Ker f^{s,j}_{\gamma(t)} \supset \bigcap_{\substack{ 0 \leq i \leq s-1 \\ 1 \leq j \leq n_i}} \Ker f^{ij}_{\gamma(t)}.
\end{equation*}

Consequently, the maps $f^{s,j}_{\gamma(t)}$ are linear combinations with $\mathcal{C}^{\infty}$ coefficients of the  $f^{i,j}_{\gamma(t)}$, $i \leq s-1$, $1 \leq j \leq n_i$ (this is a consequence of the constant rank theorem). Therefore, the equations (\ref{diff:equation}) for $1 \leq r \leq s$ give us a system of differential linear equations of order $1$ involving the $f^{i,j}_{\gamma(t)}(A_t)$, $0 \leq i \leq s-1$, $1 \leq j \leq n_i$. Since $A \in \Kill^{s}(x)$, the initial condition of this system is $0$, implying $A_t \in \Kill^{s}(\gamma(t))$ for all $t$.

\bibliographystyle{amsalpha}
\bibliography{reference.bib}
\nocite{*}

\end{document}